\documentclass[twoside, 11pt]{article}
\usepackage{amssymb, amsmath, mathrsfs, amsthm}
\usepackage{graphicx}
\usepackage{color,pict2e}
\usepackage{tikz}
\usepackage[top=2cm, bottom=2cm, left=2cm, right=2cm]{geometry}
\usepackage{float, caption, subcaption}
\usepackage{amsmath,amsthm,amsfonts,amssymb,amscd,mathrsfs}

\input{epsf}

\newcommand{\bd}{\begin{description}}
\newcommand{\ed}{\end{description}}
\newcommand{\bi}{\begin{itemize}}
\newcommand{\ei}{\end{itemize}}
\newcommand{\be}{\begin{enumerate}}
\newcommand{\ee}{\end{enumerate}}
\newcommand{\beq}{\begin{equation}}
\newcommand{\eeq}{\end{equation}}
\newcommand{\beqs}{\begin{eqnarray*}}
\newcommand{\eeqs}{\end{eqnarray*}}

\definecolor{DarkGreen}{rgb}{0.2, 0.6, 0.3}

\newcommand{\sat}{{\rm sat}}
\newcommand{\diam}{{\rm diam}}
\newcommand{\Sat}{{\rm Sat}}

\usepackage{enumitem}
\newlist{lemmaCases}{enumerate}{1}
\setlist[lemmaCases]{label=Case \arabic*:, left=0pt}

\newtheorem{theorem}{Theorem}[section]

\newtheorem{lemma}[theorem]{Lemma}

\newtheorem{corollary}[theorem]{Corollary}

\newtheorem{claim}{Claim}[]

\newtheorem{problem}{Problem}

\newtheorem{observation}{Observation}[section]
\begin{document}
\title{\textbf{On the saturation number of the kite graph } 
}

\author{ Huanying Bian\footnote{School of Mathematics and Statistics, Shandong University of Technology,
		Zibo 255000, China. {\tt
			huanying2000330@163.com}},\ \  Qing Cui\footnote{School of Mathematics,
		Nanjing University of Aeronautics and Astronautics,
		Nanjing 210016,  China. {\tt cui@nuaa.edu.cn}},\ \ Shengjin Ji\footnote{School of Mathematics and Statistics, Shandong University of Technology,
		Zibo 255000, China. {\tt
			jishengjinjin@sdut.edu.cn}} \footnote{Corresponding author: Shengjin Ji}, Fuhong Ma\footnote{School of Mathematics and Statistics, Shandong University of Technology,
		Zibo 255000, China. {\tt mafuhongsdnu@163.com}}
	\\\vspace{3mm}
	\small	}
\date{}
\maketitle
\vspace{-6pt}
\begin{abstract}
	For a fixed graph $H$, a graph $G$ is $H$-saturated if $G$ does not contain a copy of $H$, but adding any edge $e \in E(\overline{G})$ to $G$ creates a copy of $H$. The saturation number $\mathrm{sat}(n,H)$ is the minimum number of edges in an $H$-saturated graph on $n$ vertices. Let $K$ be the kite graph, formed by removing one edge from $ K_4$ and then attaching a pendant edge to a vertex of degree two in the resulting graph. 
	In this paper, 	we first establish a relationship between connectivity and 
	$K$-saturated graphs, and subsequently determine the saturation number of the kite graph  $K$. 
	Moreover, we completely characterize all extremal graphs.
	Our result provides a partial answer to a problem raised by Hua and Peng [Discrete Math. 349 (2026) 114674].
	\\[2mm]
	\\[2mm]
\end{abstract}

\section{Introduction}
All graphs considered in this paper are simple and finite. Let $G = (V(G), E(G))$, where $V(G)$ and $E(G)$ are the vertex set and the edge set of $G$, respectively. Let $|G|$ be the order of $G$ and $e(G)$ be the size of $G$. Let $N_G(v) = \{ u \in V(G) \mid uv \in E(G) \}$, $d_G(v) = |N_G(v)|$ be the degree of $v$. For a subset $V_1$ of $V(G)$, let $N_G(V_1) = \{ u \in V(G) \mid uv \in E(G) ~\text{and}~ v \in V_1\}$. 
If there exists a vertex $v \in V(G)$ such that $d_G(v)=1$, then $v$ is called a leaf, and the unique edge incident to $v$ is called a pendant edge. The distance of $u$ and $v$, denoted by $d_G(u, v)$, is the number of edges of a shortest path connecting them, in particular, if $u$ and $v$ belong to different connected components of $G$, then we define $d_G(u,v)=\infty$. We use $\diam(G) = \max_{u,v \in V(G)} d_G(u, v)$ to denote the diameter of $G$. Let $H$ be a subgraph of $G$.  We  define $d(v, H)=\max\{d_G(v,u)|\,\forall~ u\in V(H)\}$. Let $G - H$ be the subgraph $(V(G), E(G) \setminus E(H))$ of $G$. We use $G[D]$ to denote the subgraph of $G$ induced by $D \subseteq V(G)$. Let $\overline{G}$ denote the complement graph of $G$. For any two vertex-disjoint graphs $ F_1 $ and $ F_2 $, let $F_1 \cup F_2$ denote the union of $ F_1 $ and $ F_2 $, and let  $tF_1$ be the union of $t$ vertex-disjoint copies of $F_1$. The join of $ F_1 $ and $ F_2 $, denoted by $ F_1 \vee F_2 $, is the graph obtained by adding all edges between  $F_1$ and  $F_2$. Let $P_t$, $C_t$ and $K_t$  denote the path, cycle, complete graph on $t$ vertices, respectively.  Let $K_{a, b}$ be the  complete bipartite graph with parts sets of sizes $a$ and $b$. We define $C_{n-1}^+$ as the graph obtained by adding a new vertex to $C_{n-1}$ and joining it to one vertex of the cycle.



For a fixed graph $ H $, a graph $ G$ is \textit{$ H $-saturated} if$ G$ does not contain a copy of $ H $, but for any edge $ e \notin E(G) $, the graph $ G + e $ contains a copy of $ H $.
The saturation number of $H$, denoted by ${\sat}(n, H)$, is the minimum size of an $H$-saturated graph on $n$ vertices. Define $\Sat(n, H)$ as the set of all $H$-saturated graphs on $n$ vertices with $e(G) = {\sat}(n, H)$. The graphs in $\Sat(n, H)$ are called the \emph{extremal graphs}.

Over the past decades, the saturation number $\sat(n,H)$ has emerged as a compelling subject within extremal graph theory, attracting sustained research attention. The foundational result in this area, due to Erd\H{o}s, Hajnal and Moon \cite{erdos1964}, established that $\sat(n,K_p)=(p-2)(n-p+2)+\binom{p-2}{2}$, with the complete $(p-2)$-partite join $K_{p-2}\vee \overline{K}_{n-p+2}$ characterized as the unique extremal graph.
A major breakthrough came from K\'{a}szonyi and Tuza~\cite{kaszonyi1986}, whose seminal work laid the theoretical groundwork for the systematic investigation of saturated graphs and provided general bounds applicable to a wide family of host graphs.

Subsequent efforts have largely focused on determining saturation numbers for specific graphs of small order. For instance, $\sat(n,C_4)$ was first settled by Ollmann~\cite{ollmann1972}, and later streamlined by Tuza~\cite{tuza1989} with a shorter proof. The case of $C_5$ was resolved by Chen~\cite{chen2009, chen2011}, who also addressed the saturation number for $K_{2,3}$~\cite{chen2014}. More recently, Huang, Lei, Shi and Zhang~\cite{hlsz24} obtained $\sat(n, K_{3,3})$, while Lan, Shi, Wang and Zhang~\cite{lswz25} studied the saturation number of $C_6$.  In a parallel direction, Song, Hu, Ji and Cui~\cite{song2025} investigated the wheel graph $W_4$ and completely characterized its extremal graphs.
For further recent developments, we refer the reader to \cite{cameron2022,cao2023,chen2024,chen2015,chen2008,ffgj09,faudree2009,faudree2013,he2023,slp25,hlsz24,zhu2025} as well as the comprehensive survey by Currie, Faudree, Faudree, and Schmitt \cite{currie2021}.

Very recently, Hua and Peng~\cite{hua2026} determined the saturation number of the bull graph $K_1\vee (K_2\cup 2K_1)$ and  posed the following problem.
\vspace{-2mm}
\begin{problem}[Hua and Peng, \cite{hua2026}]\label{p1}
	Determine the saturation numbers for graphs on $5$ vertices which we don't know yet.
\end{problem}
\vspace{-2mm}

	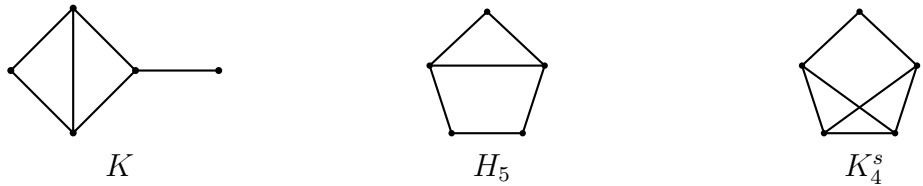
\begin{figure}[htbp]
		\centering
		\begin{minipage}[b]{0.22\textwidth}
			\centering
			\begin{tikzpicture}[scale=0.55, line join=round, line cap=round, baseline=0]
				\coordinate (A) at (-1.5,0);
				\coordinate (B) at (0,-1.5);
				\coordinate (C) at (1.5,0);
				\coordinate (D) at (0,1.5);
				\coordinate (E) at (3.5,0);
				
				\draw[black, thick] (A) -- (B) -- (C) -- (D) -- cycle;
				\draw[black, thick] (B) -- (D);
				\draw[black, thick] (C) -- (E);
				
				\foreach \p in {A,B,C,D,E} \fill[black] (\p) circle (2.4pt);
			\end{tikzpicture}
			\vspace{0.2cm}
			
			\large $K$
		\end{minipage}
		\hspace{0.8cm}
		\begin{minipage}[b]{0.22\textwidth}
			\centering
			\begin{tikzpicture}[scale=0.8, line join=round, line cap=round, baseline=0]
				\coordinate (V1) at (0,1.2);
				\coordinate (V2) at ({cos(18)}, {sin(18)});
				\coordinate (V3) at ({cos(-54)}, {sin(-54)});
				\coordinate (V4) at ({-cos(-54)}, {sin(-54)});
				\coordinate (V5) at ({-cos(18)}, {sin(18)});
				
				\draw[black, thick] (V1) -- (V2) -- (V3) -- (V4) -- (V5) -- cycle;
				\draw[black, thick] (V2) -- (V5);
				
				\foreach \p in {V1,V2,V3,V4,V5} \fill[black] (\p) circle (1.5pt);
			\end{tikzpicture}
			\vspace{0.2cm}
			
			\large $H_5$
		\end{minipage}
		\hspace{0.8cm}
		\begin{minipage}[b]{0.22\textwidth}
			\centering
			\begin{tikzpicture}[scale=0.8, line join=round, line cap=round, baseline=0]
				\coordinate (V1) at (0,1.2);
				\coordinate (V2) at ({cos(18)}, {sin(18)});
				\coordinate (V3) at ({cos(-54)}, {sin(-54)});
				\coordinate (V4) at ({-cos(-54)}, {sin(-54)});
				\coordinate (V5) at ({-cos(18)}, {sin(18)});
				
				\draw[black, thick] (V1) -- (V2) -- (V3) -- (V4) -- (V5) -- cycle;
				
				\draw[black, thick] (V2) -- (V4);
				\draw[black, thick] (V3) -- (V5);
				
				\foreach \p in {V1,V2,V3,V4,V5} \fill[black] (\p) circle (1.5pt);
			\end{tikzpicture}
			\vspace{0.2cm}
			
			\large $K_4^s$
		\end{minipage}
		
		\caption{The graphs $K$, $H_5$, and $K_4^s$.}
		\label{fig:k_h5_k4plus}
	\end{figure}

	Let $K$ be the kite graph, formed by removing one edge from $ K_4$ and then attaching a pendant edge to a vertex of degree two in the resulting graph.  (see Figure~1). Let $H_5$ be the graph on five vertices obtained from $C_5$ by adding one chord. Let $K_4^s$ be obtained from $K_4$ by replacing an edge with a path of length two. So far, from the results of \cite{hua2027a,hua2027b} by Hua and Peng together with the data in \cite[Table 1]{hua2027b}, exactly three connected graphs on 
	$5$ vertices in Problem~\ref{p1} are identified as having undetermined saturation numbers (see Figure 1).
	This naturally motivates us to pursue further investigation along this line.
	
	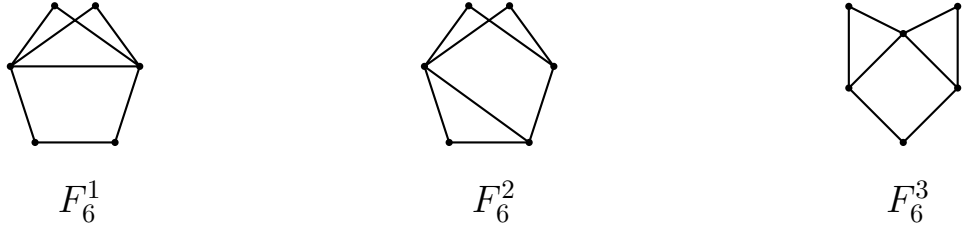
\begin{figure}[H]
		\centering
		\begin{minipage}[t]{0.24\textwidth}
			\centering
			\begin{tikzpicture}[scale=0.9, line join=round, line cap=round]
				\coordinate (T1) at (-0.3,1.2);
				\coordinate (T2) at (0.3,1.2);
				\coordinate (M1) at ({-cos(18)}, {sin(18)});
				\coordinate (M2) at ({cos(18)}, {sin(18)});
				\coordinate (B1) at ({-cos(-54)}, {sin(-54)});
				\coordinate (B2) at ({cos(-54)}, {sin(-54)});
				\draw[black, thick] (T1) -- (M1) (T1) -- (M2) (T2) -- (M1) (T2) -- (M2);
				\draw[black, thick] (M1) -- (M2);
				\draw[black, thick] (M1) -- (B1) (M2) -- (B2) (B1) -- (B2);
				\foreach \p in {T1,T2,M1,M2,B1,B2} \fill[black] (\p) circle (1.5pt);
			\end{tikzpicture}
			\vspace{0.3cm}
			
			{\Large $F^1_6$}
		\end{minipage}
		\hspace{1cm}
		\begin{minipage}[t]{0.24\textwidth}
			\centering
			\begin{tikzpicture}[scale=0.9, line join=round, line cap=round]
				\coordinate (T1) at (-0.3,1.2);
				\coordinate (T2) at (0.3,1.2);
				\coordinate (M1) at ({-cos(18)}, {sin(18)});
				\coordinate (M2) at ({cos(18)}, {sin(18)});
				\coordinate (B1) at ({-cos(-54)}, {sin(-54)});
				\coordinate (B2) at ({cos(-54)}, {sin(-54)});
				\draw[black, thick] (T1) -- (M1) (T1) -- (M2) (T2) -- (M1) (T2) -- (M2);
				\draw[black, thick] (M1) -- (B1) (M2) -- (B2) (B1) -- (B2);
				\draw[black, thick] (M1) -- (B2);
				\foreach \p in {T1,T2,M1,M2,B1,B2} \fill[black] (\p) circle (1.5pt);
			\end{tikzpicture}
			\vspace{0.3cm}
			
			{\Large $F^{2}_6$}
		\end{minipage}
		\hspace{1cm}
		\begin{minipage}[t]{0.24\textwidth}
			\centering
			\begin{tikzpicture}[scale=0.9, line join=round, line cap=round]
				\coordinate (D) at (0,-0.8);
				\coordinate (L) at (-0.8,0);
				\coordinate (R) at (0.8,0);
				\coordinate (U) at (0,0.8);
				\coordinate (E1) at (-0.8,1.2);
				\coordinate (E2) at (0.8,1.2);
				
				\draw[black, thick] (D) -- (L) -- (U) -- (R) -- cycle;
				\draw[black, thick] (E1) -- (L);
				\draw[black, thick] (E1) -- (U);
				\draw[black, thick] (E2) -- (R);
				\draw[black, thick] (E2) -- (U);
				
				\foreach \p in {D,L,R,U,E1,E2} \fill[black] (\p) circle (1.5pt);
			\end{tikzpicture}
			\vspace{0.3cm}
			
			{\Large $F^{3}_6$}
		\end{minipage}
		
		\caption{The graphs $F^1_6$, $F^2_6$, and $F^3_6$.}
	\end{figure}

	\vspace{-4mm}
	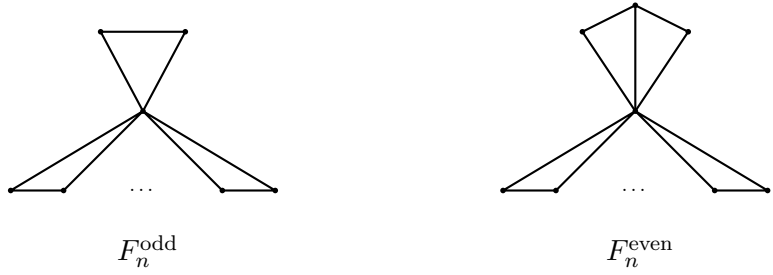
\begin{figure}[htbp]
		\centering
		\begin{minipage}[b]{0.35\textwidth}
			\centering
			\begin{tikzpicture}[scale=0.7, transform shape, line join=round, line cap=round, baseline=(current bounding box.south)]
				\coordinate (O) at (0,1.5);
				\coordinate (T1) at (-0.8,3);
				\coordinate (T2) at (0.8,3);
				\coordinate (L1) at (-2.5,0);
				\coordinate (L2) at (-1.5,0);
				\coordinate (R1) at (1.5,0);
				\coordinate (R2) at (2.5,0);
				
				\draw[black, thick] (O) -- (T1) -- (T2) -- cycle;
				\draw[black, thick] (O) -- (L1) -- (L2) -- cycle;
				\draw[black, thick] (O) -- (R1) -- (R2) -- cycle;
				
				\node at (0,0) {$\cdots$};
				
				\fill[black] (O) circle (1.5pt);
				\fill[black] (T1) circle (1.5pt);
				\fill[black] (T2) circle (1.5pt);
				\fill[black] (L1) circle (1.5pt);
				\fill[black] (L2) circle (1.5pt);
				\fill[black] (R1) circle (1.5pt);
				\fill[black] (R2) circle (1.5pt);
			\end{tikzpicture}
			\vspace{0.2cm}
			
			$F_n^{\rm odd}$
		\end{minipage}
		\hspace{0.1cm}
		\begin{minipage}[b]{0.35\textwidth}
			\centering
			\begin{tikzpicture}[scale=0.7, transform shape, line join=round, line cap=round, baseline=(current bounding box.south)]
				\coordinate (O) at (0,1.5);
				\coordinate (P1) at (-1,3);
				\coordinate (P2) at (0,3.5);
				\coordinate (P3) at (1,3);
				\coordinate (L1) at (-2.5,0);
				\coordinate (L2) at (-1.5,0);
				\coordinate (R1) at (1.5,0);
				\coordinate (R2) at (2.5,0);
				
				\draw[black, thick] (O) -- (P1) -- (P2) -- (P3) -- (O);
				\draw[black, thick] (O) -- (P2);
				\draw[black, thick] (O) -- (L1) -- (L2) -- cycle;
				\draw[black, thick] (O) -- (R1) -- (R2) -- cycle;
				
				\node at (0,0) {$\cdots$};
				
				\fill[black] (O) circle (1.5pt);
				\fill[black] (P1) circle (1.5pt);
				\fill[black] (P2) circle (1.5pt);
				\fill[black] (P3) circle (1.5pt);
				\fill[black] (L1) circle (1.5pt);
				\fill[black] (L2) circle (1.5pt);
				\fill[black] (R1) circle (1.5pt);
				\fill[black] (R2) circle (1.5pt);
			\end{tikzpicture}
			\vspace{0.2cm}
			
			$F_n^{\rm even}$
		\end{minipage}
		\caption{The graphs $F_n^{\rm odd}$ and $F_n^{\rm even}$.}
	\end{figure}
	In this paper, we determine the exact value of $\operatorname{sat}(n,K)$ for all $n \ge 5$ and give a complete characterization of the extremal graphs. Our result provides a partial answer to  Problem~\ref{p1}.
	To present our main result, we first introduce some notations for kite-saturated graphs.
	
	Let $F_6^1$, $F_6^2$, $F_6^3$ be three graphs on six vertices obtained from $H_5$ as shown in Figure~2, respectively. 
	Let $F_n^{\textup{odd}} = \left\{K_1 \vee \frac{n-1}{2}K_2 \;\middle|\; n \text{ is odd},\ n \geq 5\right\}$, $F_n^{\textup{even}} = \left\{K_1 \vee \left(\frac{n-4}{2}K_2 \cup P_3\right) \;\middle|\; n \text{ is even},\ n \geq 6\right\},$ as shown in Figure~3.

	For $n \ge 5$, we define a function $f(n)$ and a family $\mathcal{K}_n$ of graphs as follows:
	\begin{equation*}
		f(n) =
		\begin{cases}
			\frac{3n}{2} - 1, & \text{if } n \equiv 0 \pmod 4; \\[4pt]
			\frac{3n}{2} - \frac{3}{2}, & \text{if } n \equiv 1 \pmod 4; \\[4pt]
			\frac{3n}{2} - 2, & \text{if } n \equiv 2 \pmod 4; \\[4pt]
			\frac{3n}{2} - \frac{3}{2}, & \text{if } n \equiv 3 \pmod 4.
		\end{cases}
	\end{equation*}
	
	\begin{equation*}
		\texttt{and}\  \mathcal{K}_n =
		\begin{cases}
			\big\{ F_{n-4p}^{\text{even}} \cup pK_4, 0 \le p \le \frac{n-8}{4} \big\}, & \text{if } n \equiv 0 \pmod 4; \\[6pt]
			\big\{ K_1 \cup \frac{n-1}{4}K_4, \; H_5 \cup \frac{n-5}{4}K_4 \big\} \cup \big\{ F_{n-4p}^{\text{odd}} \cup pK_4,  0 \le p \le \frac{n-5}{4} \big\}, & \text{if } n \equiv 1 \pmod 4; \\[6pt]
			\big\{ K_2 \cup \frac{n-2}{4}K_4 \big\}, & \text{if } n \equiv 2 \pmod 4; \\[6pt]
			\big\{ K_3 \cup \frac{n-3}{4}K_4 \big\} \cup \big\{ F_{n-4p}^{\text{odd}} \cup pK_4, \; 0 \le p \le \frac{n-7}{4} \big\}, & \text{if } n \equiv 3 \pmod 4.
		\end{cases}
	\end{equation*}
	
	\begin{theorem}\label{thm:main}
		Let $G$ be a minimum kite-saturated graph of order $n\ge 5$. Then, we have
		$e(G)=\sat(n, K) = f(n)$,
		moreover, $\Sat(n, K) = \mathcal{K}_n$.
	\end{theorem}
	
	The remainder of this paper is organized as follows. Section~2 presents the structural properties of kite-saturated graphs and introduces the necessary definitions. In Section~3, we determine the saturation number and characterize the extremal graphs for the connected case. Section~4 analyzes the structure of disconnected kite-saturated graphs. Finally, in Section~5, we combine the results obtained in  the connected and disconnected cases to complete the proof of Theorem~\ref{thm:main}.
	
	\section{Properties of kite-saturated graphs}
	In this section, we introduce some notations that will be used throughout the paper, and establish several properties of the kite graph and kite-saturated graphs.
	
	For convenience, we now classify  the  edges and vertices of the  kite graph as shown in Figure~4. In addition, assume that $G$ is kite-saturated and $xy\notin E(G)$. Then $G+xy$ contains a copy of the kite graph for each $xy\in  E(\overline{G})$, say $K^{xy}$.  From now on, we will present some properties of $x$ and $y$.
	
	\begin{observation}\label{obs:upperbound}
		Let $G_0 \in \mathcal{K}_n$. We have that $G_0$ is kite-saturated, moreover, $\sat(n,K)\le f(n)$.
	\end{observation}
	\begin{proof}
		We now show that each graph in $\mathcal{K}_n$ is kite-saturated. To avoid redundancy in the proof, it suffices to prove for $G_0 \cong K_2 \cup \frac{n-2}{4}K_4$ without loss of generality. Each component of $G_0$ has at most $4$ vertices, so $G_0$ is $K$-free since $|V(K)|=5$. For any $xy \in E(\overline{G_0})$, assume $x$ lies in a $K_4$ component. Then $G_0 + xy$ contains a kite formed by this $K_4$ together with $y$. Hence $G_0$ is kite-saturated. Since $e(G_0)=f(n)$, we obtain $\operatorname{sat}(n,K) \le f(n)$.
	\end{proof}
	\begin{figure}[htbp]\label{defineK}
		\centering
		\tikzset{
			solid circle/.style={circle, fill=black, inner sep=0pt, minimum size=5pt},
			empty circle/.style={circle, draw=black, thick, fill=white, inner sep=0pt, minimum size=5pt},
			solid square/.style={rectangle, fill=black, inner sep=0pt, minimum size=5pt},
			empty square/.style={rectangle, draw=black, thick, fill=white, inner sep=0pt, minimum size=5pt},
			line pt/.style={circle, fill=black, inner sep=0pt, minimum size=2.5pt} 
		}
		
		\begin{tikzpicture}[scale=0.9, line join=round, line cap=round, baseline=(current bounding box.center)]
			
			\coordinate (T) at (0,1.5);       
			\coordinate (B) at (0,-1.5);      
			\coordinate (L) at (-1.5,0);      
			\coordinate (R) at (1.5,0);       
			\coordinate (E) at (3.5,0);       
			
			\draw[black, thick] (L) -- (T);
			\draw[black, thick] (L) -- (B);
			\draw[black, thick] (T) -- (R);
			\draw[black, thick] (B) -- (R);
			\draw[black, thick] (T) -- (B);
			
			\draw[black, thick, dashed] (R) -- (E);
			
			\node[solid square] at (L) {};
			\node[empty square] at (T) {};
			\node[empty square] at (B) {};
			\node[empty circle] at (R) {};
			\node[solid circle] at (E) {};
			
		\end{tikzpicture}
		\hspace{2cm}
		\begin{tikzpicture}[scale=0.8, line join=round, line cap=round, baseline=(current bounding box.center)]
			
			\draw[black, thick, dashed] (0,5) -- (1,5);
			\node[line pt] at (0,5) {};
			\node[line pt] at (1,5) {};
			\node[right] at (1.3,5) {\textup{type I edge}};
			
			\draw[black, thick] (0,4.3) -- (1,4.3);
			\node[line pt] at (0,4.3) {};
			\node[line pt] at (1,4.3) {};
			\node[right] at (1.3,4.3) {\textup{type II edge}};
			
			\node[solid circle] at (0.5,3.6) {};
			\node[right] at (1.3,3.6) {\textup{type } $\alpha_0$ \text{ vertex}};
			
			\node[empty circle] at (0.5,2.9) {};
			\node[right] at (1.3,2.9) {\textup{type } $\alpha_1$ \text{ vertex}};
			
			\node[empty square] at (0.5,2.2) {};
			\node[right] at (1.3,2.2) {\textup{type } $\alpha_2$ \text{ vertex}};
			
			\node[solid square] at (0.5,1.5) {};
			\node[right] at (1.3,1.5) {\textup{type } $\alpha_3$ \text{ vertex}};
			
		\end{tikzpicture}
		
		\caption{Classifications of vertices and edges in $K$.}
	\end{figure}
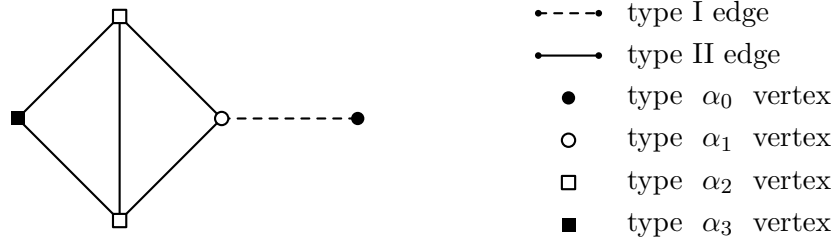
	

	\begin{observation}\label{cla:degree}
		Let $G \in \operatorname{Sat}(n, K)$ with order $n \geq 5$. Then $\delta(G) \le 2$ and $\Delta(G)\ge 3$.
	\end{observation}

	\begin{proof}
		Suppose $G \in \operatorname{Sat}(n, K)$ with $n \geq 5$. Together with the structure of $K$, we observe that $\Delta(G)\ge 3$.   We now assume that $\delta(G) \ge 3$. We thus estimate that $e(G) \geq \frac{\delta(G)}{2} n \ge \frac{3n}{2} > f(n)$, which contradicts Observation \ref{obs:upperbound}.
	\end{proof}
	\begin{observation} \label{obs:dist_type}
		If $xy$ is a type II edge in $K^{xy}$, then $d_G(x, y) = 2$. Furthermore, if $d_G(x, y) \ge 3$, then $xy$ must be the type I edge.
	\end{observation}
	
	\begin{proof}
		Let $xy$ be a type II edge of $K^{xy}$.
		Note that every type II edge in $K^{xy}$ belongs to at least one triangle. It follows that $x$ and $y$ have a common neighbor in $G$, which implies $d_G(x, y) = 2$. The second statement then follows directly from the first.
	\end{proof}
	
	\begin{lemma}\label{lem:keylemma}
		Suppose that $x$ and $y$ belong to one component of $G$.  If $xy$ is the type I edge and $d_G(y) \ge 3$, then $y$ is not a type $\alpha_1$ vertex.
	\end{lemma}
	
	\begin{proof}
		Let $xy$ be a type I edge. By assumption, we have $d_G(x) \ge 1$ and $d_G(y) \ge 3$. Write $N_G(x) = \{x_1, x_2, \dots, x_{m_1}\}$ and $N_G(y) = \{y_1, y_2, \dots, y_{m_2}\}$, where $m_1 \ge 1$ and $m_2 \ge 3$. Since $G$ is connected, there is a shortest path between $x$ and $y$; without loss of generality, let this path be $P = x v_1 v_2 \cdots v_r y$, with $v_1 = x_1$ and $v_r = y_1$ (in particular, when $r = 1$, we have $x_1 = y_1$).

		Suppose, for contradiction, that $y$ is the type $\alpha_1$ vertex in $K^{xy}$. We  observe that $d_G(y)=3$ and $|V(K^{xy})\cap N_G(y)|=3$, otherwise, we would obtain a copy of $K$ in $G$, a contradiction.  Hence $G[\{y, y_1, y_2, y_3\}] \cong K_4$, which in turn implies that $G[\{y, y_1, y_2, y_3, y^*\}]$ contains a copy of $K$, where $y^*=v_{r-1}$ if $r\ge 2$, $y^*=x$ otherwise, again a contradiction. This  completes the proof. 
	\end{proof}
	\vspace{-2mm}
	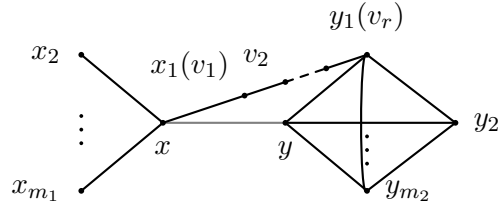
\begin{figure}[htbp]\label{keyfig}
		\centering
		\begin{tikzpicture}[scale=0.9, line join=round, line cap=round]
			
			\coordinate (x) at (0,0);
			\coordinate (y) at (1.8,0);
			
			\coordinate (x2) at (-1.2,1.0);
			\coordinate (xm1) at (-1.2,-1.0);
			
			\coordinate (y1) at (3.0,1.0);
			\coordinate (y2) at (4.3,0);
			\coordinate (ym2) at (3.0,-1.0);
			
			\coordinate (x1) at (1.2, 0.4);
			\coordinate (v1) at (1.8, 0.6);
			\coordinate (vr) at (2.4, 0.8);
			
			\draw[gray, thick] (x) -- (y);
			
			\draw[black, thick] (x) -- (x2);
			\draw[black, thick] (x) -- (xm1);
			
			\fill[black] (-1.2,0.1) circle (0.8pt);
			\fill[black] (-1.2,-0.1) circle (0.8pt);
			\fill[black] (-1.2,-0.3) circle (0.8pt);
			
			\draw[black, thick] (x) -- (x1);
			\draw[black, thick] (x1) -- (v1);
			\draw[black, thick, dashed] (v1) -- (vr);
			\draw[black, thick] (vr) -- (y1);
			
			\draw[black, thick] (y) -- (y1);
			\draw[black, thick] (y) -- (y2);
			\draw[black, thick] (y) -- (ym2);
			
			\draw[black, thick] (y1) -- (y2);
			\draw[black, thick] (y2) -- (ym2);
			
			\draw[black, thick] (y1) to[out=120, in=-120, looseness=0.3] (ym2);
			
			\fill[black] (3.0,-0.2) circle (0.8pt);
			\fill[black] (3.0,-0.4) circle (0.8pt);
			\fill[black] (3.0,-0.6) circle (0.8pt);
			
			\node[left=4pt] at (x2) {$x_2$};
			\node[left=4pt] at (xm1) {$x_{m_1}$};
			\node[below=3pt] at (x) {$x$};
			\node[below=3pt] at (y) {$y$};
			
			\node[above=5pt] at (y1) {$y_1(v_r)$};
			\node[right=3pt] at (y2) {$y_2$};
			\node[right=3pt] at (ym2) {$y_{m_2}$};
			
			\node[above left=2pt] at (x1) {$x_1(v_1)$};
			\node[above left=2pt] at (v1) {$v_2$};
			
			\foreach \p in {x,y,x2,xm1,y1,y2,ym2,x1,v1,vr}
			\fill[black] (\p) circle (1.2pt);
			
		\end{tikzpicture}
		\caption{The local structure of $G$ in Lemma~\ref{lem:keylemma}.}
	\end{figure}
	\begin{corollary}\label{cor:type_II}
		Suppose that $x$ and $y$ belong to one component of $G$.  If $d_G(x) \neq 2$ and $d_G(y) \ge 3$, then $xy$ is a type II edge of $K^{xy}$ and $d_G(x,y) = 2$.
	\end{corollary}
	
	\begin{proof}
		Suppose for contradiction that $xy \in E(K^{xy})$ is a type I edge. Then, one vertex of $x$ and $y$ is the type $\alpha_1$ vertex.  By Lemma~\ref{lem:keylemma}, $y$ with $d_G(y) \ge 3$ is not the type $\alpha_1$ vertex. Hence, $x$ is the type $\alpha_1$ vertex.
		Together with the condition $d_G(x) \neq 2$, we deduce that $d_G(x) \ge 3$. 
		But, by Lemma~\ref{lem:keylemma},  $x$ is not a type $\alpha_1$ vertex. This contradicts our assumption.
		Hence, $xy $ is a type II edge of $K^{xy}$. Moreover, we derive $d_G(x,y) = 2$ from Observation~\ref{obs:dist_type}.
	\end{proof}

	\begin{lemma}\label{lem:diam_4}
		Let $G$ be a connected kite-saturated graph. Then $\text{diam}(G) \le 4$.
	\end{lemma}
	
	\begin{proof}
		Assume to the contrary that $\text{diam}(G) \ge 5$. Let $P = v_1v_2 \cdots v_lv_{l+1}$ be a diametral path of $G$ with $l \ge 5$. Consider the graph $G + v_2v_l$, by our assumption, it contains  $K^{v_2v_l}$ as a copy of $K$. Since $P$ is an induced path, $d_G(v_2, v_l) = l - 2 \ge 3$. By Observation~\ref{obs:dist_type}, $v_2v_l$ is the type I edge of $K^{v_2v_l}$.
		By  symmetry between $v_2$ and $v_l$, we assume without loss of generality that $v_2$ is the type $\alpha_1$ vertex and $v_l$ is the type $\alpha_0$ vertex in $K^{v_2v_l}$.
		Then, $v_2$ are adjacent to two type $\alpha_2$ vertices of  $K^{v_2v_l}$. If they are $v_1$ and $v_3$, then they are adjacent from  the structure of $K^{v_2v_l}$. This contradicts the choice of $P$. Thus, $K^{v_2v_l}$ contains a neighbor of $v_2$ as a type $\alpha_2$ vertex, distinct from $v_1$ and $v_3$. Consequently, $d_G(v_2) \ge 3$.  By Lemma~\ref{lem:keylemma}, $v_2$ cannot be  the type $\alpha_1$ vertex of $K^{v_2v_l}$, which contradicts our assumption.
		Hence,  $v_2v_l$  is not a type I edge. This  contradicts our initial assumption that $\text{diam}(G) \ge 5$. Therefore,  $\text{diam}(G) \le 4$, as required.
	\end{proof}

	\section{Connected kite-saturated graphs}
	In this section, we will determine all connected kite-saturated graphs on $n$ vertices of  minimum size. We first  classify  vertices of  degree two in a connected kite-saturated graph.
	
	Let $D_2$ denote the set of vertices of degree $2$ in $G$. We partition $D_2$ into two disjoint subsets based on the adjacency of their neighbors:
	\begin{align*}
		D_2^1 &= \{v \in V(G) \mid d_G(v) = 2 \text{ and the two neighbors of } v \text{ are adjacent}\}, \\
		D_2^2 &= \{v \in V(G) \mid d_G(v) = 2 \text{ and the two neighbors of } v \text{ are non-adjacent}\}.
	\end{align*}
	
	For $n \ge 5$, we define a function $f_1(n)$ as follows:
	\begin{equation*}
		f_1(n) =
		\begin{cases}
			\frac{3n}{2} - \frac{3}{2}, & \text{if } n \text{ is odd;} \\[6pt]
			\frac{3n}{2} - 1, & \text{if } n \text{ is even.}
		\end{cases}
	\end{equation*}
	The main result of this section is presented as follows.
	\begin{theorem}\label{thm:connected_bound}
		Let $G$ be a connected kite-saturated graph of order $n \ge 5$. We have
		$e(G) \ge f_1(n).$
		Moreover, equality holds only if $G \in \{H_5, F_n^{\text{odd}}\}$  for odd $n\ge 5$ or $G \in \{F_6^1, F_6^2, F_6^3, F_n^{\text{even}}\}$ for even $n\ge 6$.
	\end{theorem}
	\subsection{ kite-saturated graphs with connectivity two}
	In this subsection,  we will present some local structures of kite-saturated graphs with respect to $D_2$ and the sizes of kite-saturated graphs with small order.  For convenience, let $C_4$ denote the unique $4$-cycle of the kite graph. If a new edge is added  to connect two nonadjacent vertices of $C_4$, then we call this new edge is a chord of $C_4$.
	\begin{lemma}\label{lem:H5_existence}
		Let $G$ be a $2$-connected kite-saturated graph of order $n$. For $n \ge 6$, if there exists $v \in D^2_2$, then there exist four vertices $v_1,v_2,v_3,v_4 \in V(G)$ such that $G[\{v, v_1, v_2, v_3, v_4\}] \cong H_5$.
	\end{lemma}
	
	\begin{proof}
		Let $v\in D^2_2$. We use $v_1$ and $v_2$ to denote the two nonadjacent neighbors of $v$.  Thus, from our assumption, $G + v_1v_2$ contains $K^{v_1v_2}$ as a copy of $K$.  We first verify the following claim.
		\setcounter{claim}{0} 
		\begin{claim}\label{v1v2}
			$v_1v_2$ is a type II edge in $K^{v_1v_2}$.   \end{claim}
		\proof
		If not, we assume that $v_1v_2$ is the type I edge of $K^{v_1v_2}$.   Without loss of generality,  let $v_2$ be the type $\alpha_1$ vertex and $v_1$ be the type $\alpha_0$ vertex in $K^{v_1v_2}$.
		Observe that  $v\notin V(K^{v_1v_2})$ by $d_G(v)=2$.
		Hence, we derive that $G$ contains a copy of $K$ by replacing $v_1v_2$ with $vv_2$, a contradiction.
		\qed
		
		Observe that  at least one of the two  $v_1$ and $v_2$ is a type $\alpha_2$ vertex. Furthermore, $C_4$ contains exactly one chord in $K^{v_1v_2}$, otherwise, we derive a copy of $K$ in $G$, a contradiction.
		We now proceed  by  considering the types of $v_1$ and $v_2$.
		
			\medskip
		\noindent\textbf{Case 1.} $v_1$ and $v_2$ are not both type $\alpha_2$ vertices.
		Then $v_1v_2\in E(C_4)$. Let the other two vertices of $C_4$ be $v_3$ and $v_4$ such that  $C_4=v_1v_2v_3v_4v_1$. Note that these four vertices form a path $v_2v_3v_4v_1$ of $G$ with an additional edge either $v_1v_3$ or $v_2v_4$. Together with $N_G(v)=\{v_1,v_2\}$, we get a $5$-cycle as $vv_2v_3v_4v_1v$ with an additional edge $v_1v_3$ or $v_2v_4$.
		Thus, $G[\{v, v_1, v_2, v_3, v_4\}] \cong H_5$.
		
			\medskip
		\noindent\textbf{Case 2.} $v_1$ and $v_2$ are both type $\alpha_2$ vertices. It follows from $d_G(v)=2$ that $v$ is neither type $\alpha_1$ vertex nor  type $\alpha_0$ vertex in $K^{v_1v_2}$. We denote these two vertices in $K^{v_1v_2}$ by $v_3$ and $v_4$, respectively.
		
		We first assume that  $v_4 \in N_G(v_1) \cup N_G(v_2)$. We claim that $v_4 \notin N_G(v_1) \cap N_G(v_2)$; otherwise $G[\{v,v_1, v_2, v_3, v_4\}]$ contains a copy of $K$, a contradiction.  Thus, without loss of generality, we may assume that $v_4 \in N_G(v_1)$ and $v_4 \notin N_G(v_2)$. We  observe that in $G$,   the four vertices $v_1, v_2, v_3, v_4$ form a $4$-cycle with a chord $v_1v_3$.   Consequently, $G[\{v, v_1, v_2, v_3, v_4\}] \cong H_5$.
		
		We now consider the case $v_4 \notin N_G(v_1) \cup N_G(v_2)$.
		Together with $N_G(v)=\{v_1,v_2\}$, we deduce that $v v_1(v_2) v_3 v_4$ is an induced path in $G$, so $d_G(v, v_4) = 3$.
		By our assumption, $G + vv_4$ contains $K^{vv_4}$ as a copy of $K$.  By Observation~\ref{obs:dist_type}, $vv_4$ is the  type I edge of $K^{vv_4}$. Since $v_1v_2 \notin E(G)$ and $d_G(v)=2$, $v$ cannot be the type $\alpha_1$ vertex in $K^{vv_4}$. Hence, $v_4$ is indeed that vertex.
		Note that $v_4 \notin N_G(v_1) \cup N_G(v_2)$. thus, $K^{vv_4}$ cannot contain both $v_1$ and $v_2$.
		\begin{itemize}
			\item $V(K^{vv_4})\cap \{v_1,v_2\}\neq \emptyset$.  Without loss of generality, we assume that $K^{vv_4}$ contains the vertex $v_1$.
			Hence, there exists a vertex $w$ distinguished from $v_2$ such that $V(K^{vv_4})=\{v, v_1, v_3, v_4, w\}$ see Figure~6. However, the existing edge $vv_1 \in E(G)$ can directly replace $vv_4$ to obtain a $K$-copy in $G$, a contradiction.
			
			\item $V(K^{vv_4})\cap \{v_1,v_2\}= \emptyset$. We first claim that $v_3\in V(K^{vv_4})$, otherwise, we will get a copy of $K$ in $G$ by replacing $vv_4$ by $v_3v_4$, a contradiction. Hence, there are two vertices $v_5$ and $v_6$ as a type $\alpha_2$ vertex and  the $\alpha_3$ vertex respectively, such that $V(K^{vv_4}) = \{v, v_3, v_4, v_5, v_6\}$.  Using the same argument to $v_4$, we show that $N_G(v_i)\cap \{v_1,v_2\}=\emptyset$ for $i=5,6$.   We thus derive that $d_G(v_5) \ge 3$ and  $d_G(v, v_5) = 3$, see Figure~6. But it leads to a contradiction to $G$ is kite-saturated. In fact, we can consider $G + vv_5$. Observation \ref{obs:dist_type}  states $vv_5$ cannot be a type II edge of $K^{vv_5}$. Furthermore, since $v$ cannot be a type $\alpha_1$ vertex as shown above and then $v_5$ cannot be a type $\alpha_1$ vertex by Lemma \ref{lem:keylemma}, $vv_5$ cannot be a type I edge. Thus, we deduce that $K^{vv_5}$ does not appear in $G + vv_5$, contradicting that $G$ is kite-saturated. 
		\end{itemize} Therefore, $G$ contains an induced $H_5$, completing the proof.
	\end{proof}
	
	\begin{figure}[htbp]
		\centering
		\begin{minipage}[b]{0.4\textwidth}
			\centering
			\begin{tikzpicture}[scale=0.6, transform shape, line join=round, line cap=round, baseline=0, font=\Large]
				\coordinate (v3) at (0, 0);
				\coordinate (v1) at (-1.5, -1.5);
				\coordinate (v2) at (-1.5, 1.5);
				\coordinate (v) at (-3, 0);
				\coordinate (w) at (0, -1.5);
				\coordinate (v4) at (2, 0);
				
				\draw[black, thick] (v) -- (v1) -- (v3) -- (v2) -- cycle;
				\draw[black, thick] (v1) -- (w);
				\draw[black, thick] (w) -- (v3);
				\draw[black, thick] (v3) -- (v4);
				\draw[black, thick] (w) -- (v4);
				
				\draw[black, thick, dashed] (v) to[bend left=25] (v4);
				
				\fill[black] (v) circle (1.5pt) node[left=3pt] {$v$};
				\fill[black] (v1) circle (1.5pt) node[below=3pt] {$v_1$};
				\fill[black] (v2) circle (1.5pt) node[above=3pt] {$v_2$};
				\fill[black] (w) circle (1.5pt) node[below=3pt] {$w$};
				\fill[black] (v3) circle (1.5pt) node[above right=1pt] {$v_3$};
				\fill[black] (v4) circle (1.5pt) node[right=3pt] {$v_4$};
			\end{tikzpicture}
			\vspace{0.2cm}
			
			(a)
		\end{minipage}
		\hspace{1.5cm}
		\begin{minipage}[b]{0.4\textwidth}
			\centering
			\begin{tikzpicture}[scale=0.6, transform shape, line join=round, line cap=round, baseline=0, font=\Large]
				\coordinate (v3) at (0, 0);
				\coordinate (v4) at (1.5, -1.5);
				\coordinate (v5) at (3, 0);
				\coordinate (v6) at (1.5, 1.5);
				\coordinate (v1) at (-1.5, -1.5);
				\coordinate (v2) at (-1.5, 1.5);
				\coordinate (v) at (-3, 0);
				
				\draw[black, thick] (v) -- (v1) -- (v3) -- (v2) -- cycle;
				\draw[black, thick] (v3) -- (v4) -- (v5) -- (v6) -- cycle;
				\draw[black, thick] (v3) -- (v5);
				
				\draw[black, thick, dashed] (v) -- (v4);
				
				\fill[black] (v) circle (1.5pt) node[left=3pt] {$v$};
				\fill[black] (v1) circle (1.5pt) node[below=3pt] {$v_1$};
				\fill[black] (v2) circle (1.5pt) node[above=3pt] {$v_2$};
				\fill[black] (v3) circle (1.5pt) node[above=3pt] {$v_3$};
				\fill[black] (v4) circle (1.5pt) node[below=3pt] {$v_4$};
				\fill[black] (v5) circle (1.5pt) node[right=3pt] {$v_5$};
				\fill[black] (v6) circle (1.5pt) node[above=3pt] {$v_6$};
			\end{tikzpicture}
			\vspace{0.2cm}
			
			(b)
		\end{minipage}
		
		\caption{The possible local structures of $G$.}
		\label{fig:combined_structures}
	\end{figure}
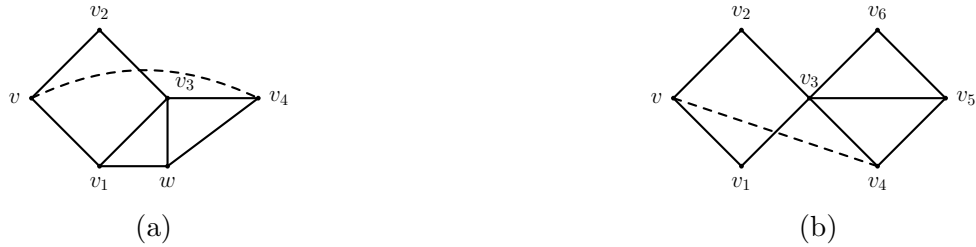
	
	\begin{lemma}\label{lem:D2_structure}
		Let $G$ be a $2$-connected kite-saturated graph of order $n\ge 6$ with $D_2^2 \neq \emptyset$. Then one of the following holds:
		
		(i) if $E(G[D_2^2])\ne \emptyset$, then $|D_2^2|=2$ and its two vertices are adjacent;
		
		(ii) if $E(G[D_2^2])= \emptyset$, then any two vertices of $D_2^2$ are nonadjacent and have exactly two common neighbors.
		%
		%
	\end{lemma}
	
	\begin{proof}
		Let $G$ be a $2$-connected kite-saturated graph of order $n \ge 6$ with $D_2^2 \neq \emptyset$. It is trivial for $|D_2^2|=1$. Thus, we now assume $|D_2^2|\ge 2$.  We first present a fundamental property for any pair of  nonadjacent vertices in $D_2^2$.
		\vspace{-2mm}
		\begin{claim}
		 If $u,v\in D_2^2$ and $uv\notin E(G)$, then $d_G(u,v)=2$ and $N_G(u) = N_G(v)$.
		\end{claim}\vspace{-2mm}
		\proof Let $u$, $v$ be two vertices of $D_2^2$ with   $uv \notin E(G)$.  Since $G$ is kite-saturated, $G + uv$ contains $K^{uv}$ as a copy of $K$.  By the choice of $u,v$, neither $u$ nor $v$ does not  a type $\alpha_1$ vertex in $K^{uv}$. Thus, $uv$ is a type II edge. It follows from Observation~\ref{obs:dist_type} that  $d_G(u, v) = 2$.  Hence, $u$ and $v$  belong to $C_4$ of $K^{uv}$. 
		
		We next show that they share  two common neighbors, i.e., $N_G(u) = N_G(v)$.   Observe that $u$ and $v$ share at least one common neighbor, say $w_1$.  Assume to the contrary that $N_G(u) \neq N_G(v)$. Together with their degree,  $w_1$ is indeed their unique common neighbor. This leads to $w_1$ belonging to $V(C_4)$. Hence, either $uw_1$ or $vw_1$ is the unique chord of $C_4$~(If $uw$ and $vw$  lie on $C_4$, then  $uv$ would be the chord of $C_4$ in $K^{uv}$. But this contradicts to our assumption.). By symmetry, without loss of generality, assume that $uw_1$ is the chord. But this brings that two neighbors of $u$ are adjacent, which contradicts the assumption $u\in D^2_2$.\qed
		By Claim 1, statement (ii) holds. So we may assume that $E(D_2^2)\neq \emptyset$; then there exist two vertices $u,v\in V(D_2^2)$ such that $uv\in E(G)$. If $|D_2^2|=2$, then (i) holds. Hence, assume that $|D_2^2|>2$. Let $w$ be the third vertex of $D_2^2$, distinct from $u$ and $v$. Together with the definition of $D_2^2$, we deduce that $|N_G(w)\cap \{u,v\}|\le 1$.
		Assume that $u\notin N_G(w)$. Then, by Claim 1, $u$ and $w$ have two common neighbors. Hence, $v\in N_G(w)$. It follows that $v$ is a common neighbor of $u$ and $w$. Let $w'$ be another common neighbor of them. Since $n\ge 6$, we also get $d_G(w')\ge 3$. In fact, considering $G+vw'$ yields a contradiction, as this addition does not create a copy of $K$. Thus, statement (i) is also true.
		
		Therefore, this completes the proof.
	\end{proof}

	\begin{lemma}\label{lem:n5_exact}
		Let $G$ be a $2$-connected kite-saturated graph with five vertices. Then we have $e(G) \ge f_1(5)$ with equality if and only if $G \cong H_5$.
	\end{lemma}
	
	\begin{proof}
		{}Let $G$ be a 2-connected kite-saturated graph of order $n=5$. We claim that $G$ must contain a $C_3$. Otherwise, by the property of triangle-free graphs and the $2$-connectivity of $G$, $G$ is isomorphic to either $K_{2,3}$ or $C_5$. However, it is verified that neither $K_{2,3}$ nor $C_5$ is kite-saturated. Thus, $G$ contains a $C_3$.
		
		Let $y_1, y_2, y_3$ be the three vertices of $C_3$ in $G$, and let $y_4$ and $y_5$ be the remaining two vertices of $G$. Since $G$ is 2-connected, its minimum degree is at least two.
		
		If $d_G(y_i) \ge 3$ for all $i \in \{1, 2, 3\}$, then the sum of degrees in $G$ is at least $ 13$. This implies that $e(G) \ge \lceil 13/2 \rceil = 7 > f_1(5) = 6$.  Thus, we may assume that at least one vertex of the triangle, say $y_1$, has degree two, which from the connectivity of $G$ infers that $d_G(y_1,y_i)=2$ for $i=4,5$.
		
		If $y_2$ and $y_3$ have a common neighbor distinct from $y_1$, then by the choice of $G$, we deduce that $G\cong 3 K_{1}\vee K_2$, otherwise, $G$ contains a copy of $K$. This leads to $e(G)=7>6$.
		If $y_1$ is the unique common neighbor of $y_2$ and $y_3$, then $y_4y_5\in E(G)$. It follows that  $G\cong H_5$. 
	\end{proof}
	
	\begin{lemma}\label{lem:n6_exact}
		Let $G$ be a $2$-connected kite-saturated graph with six vertices Then we have $e(G) \ge f_1(6)$ with equality if and only if $G \in \{F_6^1, F_6^2, F_6^3\}$.
	\end{lemma}
	
	\begin{proof}
		Suppose that $G$ is a 2-connected kite-saturated graph of order $n = 6$. We first establish the diameter of $G$. 
		Lemma \ref{lem:diam_4} infers $\text{diam}(G) \le 4$.
		If $\text{diam}(G) = 4$, then there is an induced path with length four, say $P_5=v_1v_2v_3v_4v_5$. Hence, $V(G)\setminus V(P_5)$ contains exactly one vertex, say $v_0$. Note that $d_G(v_5),d_G(v_1)\ge2$. It results in  $v_0v_1,v_0v_5\in E(G)$. But, this leads to a contradiction to $\text{diam}(G) = 4$. Thus, we  arrive at $\text{diam}(G) \in \{2, 3\}$.
		
		Observe that $e(G) \ge \lceil (6 \times 3) / 2 \rceil = 9 > f_1(6)$  if $\delta(G) \ge 3$. Hence the remaining case is that there exist at least one vertex with degree two.  Let $v$ be a vertex of $G$ with $d_G(v) = 2$.
		
		If $v \in D_2^2$, then by Lemma 3.2, $G$ contains an induced subgraph $H_5$. Let $H_5 \cong G[\{v, v_1, v_2, v_3, v_4\}]$. Consider the remaining vertex, say $v_5$. Since $e(H_5) = 6$ and $G$ is 2-connected, $v_5$ must have at least two neighbors in $H_5$. If $d_G(v_5) \ge 3$, then $e(G) \ge 6 + 3 = 9 > f_1(6)$. Hence, we must have exactly $d_G(v_5) = 2$. Together with the established bound $\text{diam}(G) \in \{2, 3\}$ and the property of the kite-saturated graph, by verifying the possible connections of $v_5$ to the vertices of $H_5$ it is easily verified that equality holds if and only if $G \in \{F_6^1, F_6^2, F_6^3\}$.
		
		From now on, we assume that all vertices of degree two belong to $D^1_2$. Let $u \in D^1_2$. Thus there is a $C_3$ containing $u$, label its three vertices as $u, v_1, v_2$ and the remaining vertices as $v_3, v_4, v_5$. Since any path from $u$ to $v_i$ ($i \in \{3,4,5\}$) must pass through the cut-set $\{v_1, v_2\}$, we deduce that $2 \le d_G(u, v_i) \le 3$. If $G[\{v_3, v_4, v_5\}]$ contains no edges, then, from $2$-connectedness of $G$, we derive that $G \cong K_2 \vee 4K_1$. But this leads to $e(G) > f_1(6)$. Consequently, we assume that $G[\{v_3, v_4, v_5\}]$ contains at least one edge. It follows that $G[\{v_3, v_4, v_5\}] \cong K_2 \cup K_1, P_3$ or $C_3$. Together with $\text{diam}(G) \in \{2, 3\}$, it is not difficult to check that either $e(G) > f_1(6)$ or $G$ contains a copy of $K$.
		
		The proof is therefore complete.
	\end{proof}
	
	\begin{lemma}\label{lem:n78_exact}
		Let $G$ be a $2$-connected kite-saturated graph of order $n\in \{7, 8\}$. Then $e(G) > f_1(n)$.
	\end{lemma}
	
	\begin{proof}
		
		Note that $G$ is $2$-connected. Hence, $\delta(G) \ge 2$.
		
		\medskip
		\noindent\textbf{Case 1.} $n = 7$. In this case, $f_1(7) = 9$. Suppose that $G$ is a 2-connected kite-saturated graph with seven vertices. If $D_2^2 \neq \emptyset$, then there is a vertex $v \in D_2^2$. From Lemma \ref{lem:H5_existence}, we obtain that $G$ contains an induced subgraph $H_5$ including the vertex $v$. For short, let $V(H_5) = \{v, v_1, v_2, v_3, v_4\}$ such that $vv_1v_2v_3v$ is the 4-cycle and $v_4$ is the vertex with two neighbors $v_2$ and $v_3$ in $H_5$. Set $A = V(G) \setminus V(H_5) = \{y_1, y_2\}$. Since $e(H_5) = 6$ and $f_1(7) = 9$, in order to verify $e(G) > f_1(7)$, it suffices to check $e(G) \ge 10$. Note that $e(A, H_5) \ge 2$ by the connectivity of $G$. Hence, we have  $$e(G) \ge e(H_5) + e(A, H_5) + e(G[A]) \ge 6 + 2 + e(G[A]).$$ Evidently, $G[A] \in \{2K_1, K_2\} $. Observe that if $G[A] \cong 2K_1$, then, $e(G) \ge 10$ holds by the assumption that $G$ is $2$-connected. We now consider the case $G[A] \cong K_2 $. Assume by contradiction that $e(G) \le 9$.  By the connectivity of $G$ and $e(G) = 9$, we get that two vertices of  $A$ belong to $D_2^2$. But, this yields a contradiction to (i) of Lemma \ref{lem:D2_structure}.
		
		Now we assume that all vertices of degree two belong to $D^1_2$. Then there is a vertex $u \in D^1_2$ which is contained in a $C_3$. For short, let $V(C_3) = \{u, v_1, v_2\}$, and set $A = V(G) \setminus V(C_3) = \{y_1, y_2, y_3, y_4\}$. Since $e(C_3) = 3$ and $f_1(7) = 9$, in order to verify $e(G) > f_1(7)$, it suffices to check $e(G) \ge 10$. Note that $e(A, C_3) \ge 2$ by the connectivity of $G$, which infers that 
		$$e(G) \ge e(C_3) + e(A, C_3) + e(G[A]) \ge 3 + 2 + e(G[A]).$$ If $e(G[A]) \ge 5$, then conclusion is trivial. Consequently, we assume $e(G[A]) \le 4$. Evidently, $G[A] \in \{C_4, C_3^+, P_4, K_{1,3}, K_3\cup K_1, P_3 \cup K_1, 2K_2, K_2 \cup 2K_1, 4K_1\}$. 
		
		Observe that if $G[A] \cong K_2 \cup 2K_1, 4K_1, K_3\cup K_1$, then, $e(G) \ge 10>f_1(7)$ holds by the assumption that $G$ is 2-connected.
		In the remaining case $G[A] \in \{C_4, C_3^+, P_4,  K_{1,3}, P_3 \cup K_1, 2K_2\}$, assume by contradiction that $e(G) \le 9$. If $G[A] \in \{C_4, C_3^+, K_{1,3}, P_4, P_3 \cup K_1, 2K_2\}$, then, by our assumption and the connectivity of $G$, $e(G) = 9$. But this leads to that there is a vertex of $A$ belonging to $D_2^2$. Hence,  we get  a contradiction to our assumption that all vertices of degree two belong to $D^1_2$.
		
		\medskip
		\noindent
		\textbf{Case 2.} $n = 8$. The same argument as in the case $n=7$ can be applied to the case  to show $e(G) > f_1(n)$. Since the verification is rather tedious, it is omitted here.
	\end{proof}
	
	\begin{lemma} \label{lem:n9}
		Let $G$ be a $2$-connected kite-saturated graph of order $n=9,10$ with $D^2_2\neq\emptyset$. Then $e(G)> f_1(n)$.
	\end{lemma}
	\begin{proof}
		Suppose that $G$ is a $2$-connected kite-saturated graph with nine vertices and  $D^2_2\neq\emptyset$. Hence, there is a vertex $v\in D^2_2$. From Lemma \ref{lem:H5_existence}, we obtain that $G$ contains an induced subgraph $H_5$ including the vertex $v$. For short, let $V(H_5)=\{v,v_1,v_2,v_3,v_4\}$ such that $vv_1v_2v_3v$ is the $4$-cycle  and $v_4$ is the vertex with two neighbors $v_2$ and $v_3$ in $H_5$.  Clearly, $e(H_5)=6$.
		
		We now consider the case $n=9$, Set $A=V(G)\setminus V(H_5)=\{y_1,y_2,y_3,y_4\}$.  Since $e(H_5)=6$ and $f_1(9)=12$, in order to verify $e(G)>f_1(9)$, it suffices  to check $e(G)\ge 13$. Note that $e(A,H_5)\ge 2$ by the connectivity of $G$. So $e(G)\ge e(H_5)+e(A,H_5)+e(G[A])\ge 6+2+e(G[A])$. If $e(G[A])\ge 5$, then conclusion is trivial. Consequently, we assume $e(G[A])\le 4$. Evidently, $G[A]\in \{C_4,C^+_3, K_{1,3}, P_4, K_3\cup K_1, P_3\cup K_1,2K_2,K_2\cup 2K_1,4K_1\}$. 
		
		Observe that if $G[A]\cong K_3\cup K_1, K_2\cup 2K_1,4K_1$, then, $e(G)\ge 13$ holds by the assumption that $G$ is $2$-connected. In the remaining case $G[A]\in \{C_4,C^+_3,P_4, K_{1,3}, P_3\cup K_1,2K_2\}$, assume by contradiction that $e(G)\le 12$. then, by our assumption and the connectivity of $G$, we obtain that $11\le e(G)\le12$ and there is a vertex $y_i$ in $A$ belonging to $D^2_2$. Clearly, $vy_i\notin E(G)$ and $N_G(y_i)\ne N_G(v)$.  But, this yields a contradiction to Lemma \ref{lem:D2_structure}. 
		
		In the case $n=10$, one can follow the same argument  in the case $n = 9$ to verify $ e(G)> f_1(n)$. But, the verification is rather tedious and is thus omitted here.
		
		Therefore, the proof is complete.
	\end{proof}

	\subsection{Kite-saturated graphs with diameter two}
	\begin{theorem}[\cite{bollobas1978}]\label{thm:diam2_bound}
		Let $G$ be a $2$-connected graph of order $n$. If $\text{diam}(G) = 2$, then $e(G) \ge 2n - 5$.
	\end{theorem}

	\begin{lemma}\label{lem:1conn_exact}
		Let $G$ be a connected kite-saturated graph of order $n\ge5$
		with a cut vertex. If $\operatorname{diam}(G)=2$, then
		$e(G)\ge f_1(n)$, with equality if and only if either
		$G\cong F_n^{\mathrm{odd}}$ for odd $n$ or
		$G\cong F_n^{\mathrm{even}}$ for even $n$.
	\end{lemma}
	
	\begin{proof}
		Among all connected kite-saturated graphs of order $n$ with
		$\operatorname{diam}(G)=2$ and a cut vertex, choose $G$ to have
		the minimum number of edges. Let $v$ be a cut vertex of $G$.
		
			Since $v$ is a cut vertex, any two vertices in distinct components
		of $G-v$ have distance $2$. As $\operatorname{diam}(G)=2$, every
		vertex of $G-v$ is adjacent to $v$. Hence
$		N_G(v)=V(G)\setminus\{v\}$,
		and $v$ is a universal vertex. Denote by $G_1,G_2,\ldots,G_r$ the components of $G-v$.
		We next determine the possible structures of these components.
		
		First, no component $G_i$ contains a path of length at least $3$.
		Otherwise, suppose that
		$	x_1x_2x_3x_4
		$ is a path of length $3$ in $G_i$. Since $v$ is adjacent to all
		vertices of $G_i$, the subgraph induced by $
		\{v,x_1,x_2,x_3,x_4\}
		$
		contains a copy of the kite $K$, where $
		C_4=vx_2x_3x_4v$ and $vx_3$ is its chord. 
		This contradicts the fact that $G$ is kite-saturated.
		
		Similarly, no component $G_i$ contains a $K_3$. Indeed, suppose
		that $x_1x_2x_3x_1$ is a triangle in $G_i$. Since $G-v$ has at
		least two components, there exists a vertex $y$ in another
		component of $G-v$. Since $v$ is universal,
		$
		G[\{v,x_1,x_2,x_3,y\}
		]$
		contain a copy of $K$. This again
		contradicts the kite-saturation of $G$.
		
		Consequently, every component $G_i$ is connected, triangle-free,
		and $P_4$-free. Hence each $G_i$ is a star, and therefore
	$G_i\in\{K_1,K_2,P_3,K_{1,s}\}
		$ for some $s\ge3$.
		
		We next show that $G-v$ contains no component isomorphic to $K_1$.
		Clearly, there can be at most one such component. Suppose that
		$G_1\cong K_1$, say $V(G_1)=\{v_1\}$.	Let $G_2$ be another component of $G-v$.
		If $G_2\cong K_2$, say $G_2=v_2v_3$, then $v_1v_2\notin E(G)$.
		Since $G$ is kite-saturated, adding the edge $v_1v_2$ must create
		a copy of $K$. However, $G+v_1v_2$ contains no copy of $K$,
		a contradiction.	If $G_2\cong P_3$, let $v^*$ be the vertex of degree $2$ in
		$P_3$. Then $v_1v^*\notin E(G)$, but $G+v_1v^*$ contains no
		copy of $K$, again contradicting the saturation of $G$.
		Finally, suppose that $G_2\cong K_{1,s}$ for some $s\ge3$,
		and let $v^*$ be its center. Then $v_1v^*\notin E(G)$, whereas
		$G+v_1v^*$ contains no copy of $K$, which is again a contradiction.
		Therefore, $G-v$ contains no component isomorphic to $K_1$.
		
		Let $p_1,p_2,p_3$ denote the numbers of components of $G-v$
		isomorphic to $K_2$, $P_3$, and $K_{1,s}$, respectively.
		To allow different star components to have different numbers of
		leaves, write them as $
		K_{1,s_1},K_{1,s_2},\ldots,K_{1,s_{p_3}}$ with $ s_j\ge3.
		$
		Thus
		$
		G-v\cong
		p_1K_2\cup p_2P_3\cup
		\bigcup_{j=1}^{p_3}K_{1,s_j}.
		$
		Consequently, $
		n
		=
		1+2p_1+3p_2+
		\sum_{j=1}^{p_3}(s_j+1),
		$
		and, since $v$ is universal,
	$
		e(G)
		=
		n-1+p_1+2p_2+
		\sum_{j=1}^{p_3}s_j.
$ We now show that no star component can occur.
\setcounter{claim}{0} 
		\vspace{-2mm}
		\begin{claim}\label{claim:p3zero}
			$p_3=0$.
		\end{claim}
		\vspace{-2mm}
		\proof
		Suppose, to the contrary, that $p_3\ge1$. Choose one star
		component $K_{1,s}$ of $G-v$.
		If $s$ is odd, replace this component by $
		\frac{s+1}{2}K_2.
		$
		However, the number of edges decreases from $s$ to $
		\frac{s+1}{2}<s$.
		If $s$ is even, then $s\ge4$. Replace $K_{1,s}$ by $
		\frac{s-2}{2}K_2\cup P_3.
		$
		The number of edges decreases from $s$ to $
		\frac{s-2}{2}+2
		=
		\frac{s+2}{2}
		<s$.

		Let $G'$ be the graph obtained from $G$ by this replacement.
		Then $|V(G')|=|V(G)|$ and $e(G')<e(G)$.
		Moreover, $v$ remains a universal vertex of $G'$, and every
		component of $G'-v$ is isomorphic to $K_2$ or $P_3$.
			We claim that $G'$ is kite-saturated. Let $xy$ be any non-edge
		of $G'$. If $x$ and $y$ belong to distinct components of $G'-v$, then
		each of these components contains an edge. Choose a neighbor
		$x'$ of $x$ in the component containing $x$, and a neighbor
		$y'$ of $y$ in the component containing $y$. After adding $xy$,
	the subgraph induced by	the vertices $
		\{v,x,y,x',y'\}
		$
		contain a copy of $K$.
		If $x$ and $y$ belong to the same component, then this component
		must be a $P_3$ or $K_{1,s_j}$ with $s_j\ge 3$, and then $x$ and $y$ lie on a path, say $xzy$. Choose a vertex $w$ in another component. After adding $xy$, 
		the subgraph indueced by 
		$\{v,x,y,z,w\}$ contains a copy of $K$.
		Hence, $G'$ is kite-saturated. This contradicts the choice of $G$. Hence $p_3=0$.
		\qed
		
		Hence, from Claim \ref{claim:p3zero}, we have $
		G-v\cong p_1K_2\cup p_2P_3$.
		We next show that at most one component of $G-v$ is isomorphic
		to $P_3$.
		
		\vspace{-2mm}
		\begin{claim}\label{claim:p2one}
			$p_2\le1$.
		\end{claim}
		\vspace{-2mm}
		\proof
		Suppose, to the contrary, that $p_2\ge2$. Choose two components
		isomorphic to $P_3$ and replace them by three components
		isomorphic to $K_2$. Let $G'$ be the resulting graph.
		Since $
		|V(P_3\cup P_3)|=6=|V(3K_2)|$
.		The number of edges, however, decreases from $4$ to $3$.
		Hence
		$
		e(G')=e(G)-1<e(G).
		$ Moreover, $v$ remains universal in $G'$, and every component of
		$G'-v$ is isomorphic to $K_2$ or $P_3$. Exactly as in the proof
		of Claim~\ref{claim:p3zero}, every non-edge of $G'$ creates a
		copy of $K$ when added. Thus $G'$ is kite-saturated.
		
		This contradicts the minimality of $G$. Therefore, $
		p_2\le1.
		$
		\qed
		
		By Claims~\ref{claim:p3zero} and~\ref{claim:p2one}, $
		G-v\cong p_1K_2\cup p_2P_3$ with 
		$p_2\le1$.
		Hence $
		n-1=2p_1+3p_2$.
		If $n$ is odd, then $n-1$ is even. Thus $p_2$ is even.
		Since $p_2\le1$, we obtain $
		p_2=0.$
		It follows that $
		p_1=\frac{n-1}{2}$,
		and consequently $
		G\cong
		K_1\vee\frac{n-1}{2}K_2
		=:F_n^{\mathrm{odd}}.
		$
		Therefore, $
		e(G)
		=
		n-1+\frac{n-1}{2}
		=
		\frac{3n-3}{2}
		=
		f_1(n).
	$
		
		If $n$ is even, then $n-1$ is odd. Hence $p_2$ is odd.
		Since $p_2\le1$, we obtain $
		p_2=1$.
		Thus $
		p_1=\frac{n-4}{2}$,
		and hence $
		G\cong
		K_1\vee
		\left(
		\frac{n-4}{2}K_2\cup P_3
		\right)
		=:F_n^{\mathrm{even}}$. 
		In this case, $
		e(G)
		=
		n-1+\frac{n-4}{2}+2
		=
		\frac{3n}{2}-1
		=
		f_1(n)$.
		
		Thus the minimum possible number of edges is $f_1(n)$.
		Consequently, every connected kite-saturated graph of order $n$
		with a cut vertex and diameter $2$ satisfies
		$
		e(G)\ge f_1(n),
		$
		with equality if and only if
	$
		G\cong F_n^{\mathrm{odd}} \ \text{for odd }n,
		$
		or
		$
		G\cong F_n^{\mathrm{even}}\ \text{for even }n$.
		This completes the proof.

	\end{proof}

	\begin{lemma}\label{lem:diam2_n9}
		Suppose that $G$ is a $2$-connected kite-saturated graph  with  $n \ge 9$ and $\text{diam}(G) = 2$. Then, $e(G) > f_1(n)$.
	\end{lemma}
	
	\begin{proof}
		By Theorem~\ref{thm:diam2_bound},  we get $e(G) \ge 2n - 5$. Observe that $2n - 5 > f_1(n)$ holds for all $n \ge 9$. Thus, the result follows immediately.
	\end{proof}

	\subsection{Kite-saturated graphs with diameter three or four}

	
	\begin{lemma}\label{lem:diam34_2connected}
		Let $G$ be a kite-saturated graph of order $n$. If
		$3\le \operatorname{diam}(G)\le4$, then $G$ is $2$-connected.
	\end{lemma}
	
	\begin{proof}
		Suppose, to the contrary, that $G$ has a cut vertex $v$. Let
		$H_1,H_2,\ldots,H_k$ be the components of $G-v$. For each $i$,
		let
		$d_G(v,H_i)=\max\{d_G(v,x):x\in V(H_i)\}$.
		Since $3\le\operatorname{diam}(G)\le4$, there is a component,
		say $H_1$, such that $2\le d_G(v,H_1)\le3$.
			Let $B=V(H_1)$ and
		$A=\bigcup_{i=2}^kV(H_i)$. Define
		$A_i=\{x\in A:d_G(v,x)=i\}$ and
		$B_j=\{y\in B:d_G(v,y)=j\}$. Since there are no edges between
		$A$ and $B$, we have $d_G(x,y)=i+j$ for every
		$x\in A_i$ and $y\in B_j$. Moreover, since
		$3\le\operatorname{diam}(G)\le4$, the possible distance
		partitions are
		$(A_1,\{v\},B_1,B_2)$,
		$(A_1,\{v\},B_1,B_2,B_3)$, and
		$(A_2,A_1,\{v\},B_1,B_2)$.
		For convenience, write
		$A_i=\{x_{i1},x_{i2},\ldots,x_{i|A_i|}\}$ and
		$B_j=\{y_{j1},y_{j2},\ldots,y_{j|B_j|}\}$. Hence, the following three cases are necessary.
		
		\medskip
		\noindent
		\textbf{Case 1.}
		$V(G)$ has distance partition
		$(A_1,\{v\},B_1,B_2)$.
		
		Consider $G+x_{11}y_{21}$. Since
		$d_G(x_{11},y_{21})=3$, Observation~\ref{obs:dist_type}
		implies that $x_{11}y_{21}$ is the type I edge of
		$K^{x_{11}y_{21}}$. Without loss of generality, let $x_{11}$
		be a type $\alpha_1$ vertex. By the structure of a type I kite,
		there is a type $\alpha_2$ vertex $x_{12}\in A_1$ in
		$K^{x_{11}y_{21}}$. Hence $d_G(x_{12})\ge3$.
		
		Now consider $G+x_{12}y_{21}$. Since
		$d_G(x_{12},y_{21})\ge3$, Observation~\ref{obs:dist_type}
		implies that $x_{12}y_{21}$ is type I. As
		$d_G(x_{12})\ge3$, Lemma~\ref{lem:keylemma} implies that
		$y_{21}$ is the type $\alpha_1$ vertex of
		$K^{x_{12}y_{21}}$.
		If
		$N_G(y_{21})\cap B_2\cap V(K^{x_{12}y_{21}})=\emptyset$,
		then Figure~\ref{fig:layer_structures_mini_grid}(1), (2), and
		(4) yield a copy of $K$ in $G$, a contradiction. Hence there
		is a vertex $y_{22}\in B_2$ such that
		$y_{21}y_{22}\in E(K^{x_{12}y_{21}})$.
		If $y_{22}$ is a type $\alpha_3$ vertex, then
		Figure~\ref{fig:layer_structures_mini_grid}(3) yields a copy
		of $K$ in $G$, a contradiction. Hence $y_{22}$ is a type
		$\alpha_2$ vertex, and so $d_G(y_{22})\ge3$. Thus the only
		remaining local configuration is Figure~\ref{fig:layer_structures_mini_grid}(5).
		
		Consider $G+x_{12}y_{22}$. Since
		$d_G(x_{12},y_{22})=3$, Observation~\ref{obs:dist_type}
		implies that $x_{12}y_{22}$ is type I. Since
		$d_G(x_{12})\ge3$ and $d_G(y_{22})\ge3$,
		Lemma~\ref{lem:keylemma} implies that neither $x_{12}$ nor
		$y_{22}$ is a type $\alpha_1$ vertex, contradicting the
		definition of a type I edge. Hence Case 1 is impossible.
		
		\medskip
		\noindent
		\textbf{Case 2.}
		$V(G)$ has distance partition
		$(A_1,\{v\},B_1,B_2,B_3)$.
		
		Consider $G+x_{11}y_{31}$. Since
		$d_G(x_{11},y_{31})=4$, Observation~\ref{obs:dist_type}
		implies that $x_{11}y_{31}$ is type I. The configurations in
		Figure~\ref{fig:extended_layer_grid} yield a type $\alpha_2$
		vertex $x_{12}\in A_1$ with $d_G(x_{12})\ge3$, and a vertex
		$y_{2j}\in B_j$, where $j\in\{2,3\}$, with
		$d_G(y_{2j})\ge3$. Since
		$d_G(x_{12},y_{2j})=j+1\ge3$, Observation~\ref{obs:dist_type}
		implies that $x_{12}y_{2j}$ is type I. By
		Lemma~\ref{lem:keylemma}, neither $x_{12}$ nor $y_{2j}$ is a
		type $\alpha_1$ vertex, a contradiction. Thus Case 2 is
		impossible.
		
		\medskip
		\noindent
		\textbf{Case 3.}
		$V(G)$ has distance partition
		$(A_2,A_1,\{v\},B_1,B_2)$.
		
		Consider $G+x_{21}y_{21}$. Since
		$d_G(x_{21},y_{21})=4$, Observation~\ref{obs:dist_type}
		implies that $x_{21}y_{21}$ is type I. The configurations in
		Figure~\ref{fig:extended_layer_grid} yield vertices
		$x_{22}\in A_2$ and $y_{22}\in B_2$ such that
		$d_G(x_{22})\ge3$ and $d_G(y_{22})\ge3$. Since
		$d_G(x_{22},y_{22})=4$, Observation~\ref{obs:dist_type}
		implies that $x_{22}y_{22}$ is type I. By
		Lemma~\ref{lem:keylemma}, neither $x_{22}$ nor $y_{22}$ is a
		type $\alpha_1$ vertex, contradicting the definition of a
		type I edge. Thus Case 3 is impossible.
		
		Therefore $G$ has no cut vertex, and hence $G$ is
		$2$-connected.
	\end{proof}


		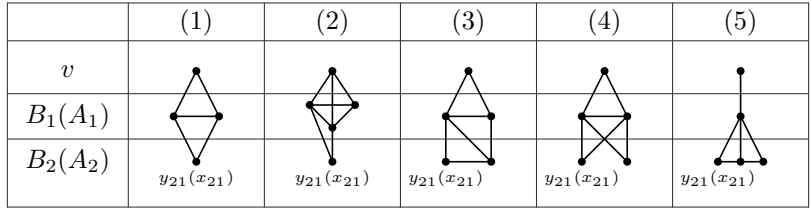
\begin{figure}[htbp]
		\centering
		\begin{tikzpicture}[line join=round, line cap=round, semithick]
			\begin{scope}[thin, black!80]
				\draw (-2.5, 2.1) -- (8.1, 2.1);
				\draw (-2.5, 1.6) -- (8.1, 1.6);
				\draw (-2.5, 0.9) -- (8.1, 0.9);
				\draw (-2.5, 0.3) -- (8.1, 0.3);
				\draw (-2.5, -0.6) -- (8.1, -0.6);
				
				\draw (-2.5, 2.1) -- (-2.5, -0.6);
				\draw (-0.9, 2.1) -- (-0.9, -0.6);
				\draw (0.9, 2.1) -- (0.9, -0.6);
				\draw (2.7, 2.1) -- (2.7, -0.6);
				\draw (4.5, 2.1) -- (4.5, -0.6);
				\draw (6.3, 2.1) -- (6.3, -0.6);
				\draw (8.1, 2.1) -- (8.1, -0.6);
			\end{scope}
			
			\node[font=\small] at (-1.7, 1.2) {$v$};
			\node[font=\small] at (-1.7, 0.6) {$B_1(A_1)$};
			\node[font=\small] at (-1.7, 0) {$B_2(A_2)$};
			
			\node[font=\small] at (0, 1.85) {(1)};
			\node[font=\small] at (1.8, 1.85) {(2)};
			\node[font=\small] at (3.6, 1.85) {(3)};
			\node[font=\small] at (5.4, 1.85) {(4)};
			\node[font=\small] at (7.2, 1.85) {(5)};
			
			\tikzset{pt/.style={circle, fill, inner sep=0pt, minimum size=3pt}}
			
			\begin{scope}[shift={(0,0)}]
				\coordinate (T) at (0, 1.2);
				\coordinate (M1) at (-0.3, 0.6);
				\coordinate (M2) at (0.3, 0.6);
				\coordinate (B) at (0, 0);
				
				\draw (M1) -- (M2);
				\draw (T) -- (M1); \draw (T) -- (M2);
				\draw (M1) -- (B); \draw (M2) -- (B);
				
				\node[pt] at (T) {}; \node[pt] at (M1) {}; \node[pt] at (M2) {}; \node[pt] at (B) {};
				\node[anchor=north, inner sep=1pt, font=\tiny] at ([yshift=-2pt]B) {$y_{21}(x_{21})$};
			\end{scope}
			
			\begin{scope}[shift={(1.8,0)}]
				\coordinate (U) at (0, 1.2);
				\coordinate (T1) at (-0.3, 0.75);
				\coordinate (T2) at (0.3, 0.75);
				\coordinate (T3) at (0, 0.45);
				\coordinate (B) at (0, 0);
				
				\draw (T1) -- (T2) -- (T3) -- cycle;
				\draw (U) -- (T1); \draw (U) -- (T2); \draw (U) -- (T3);
				\draw (T1) -- (B); \draw (T3) -- (B);
				
				\node[pt] at (U) {}; \node[pt] at (T1) {}; \node[pt] at (T2) {}; \node[pt] at (T3) {}; \node[pt] at (B) {};
				\node[anchor=north, inner sep=1pt, font=\tiny] at ([yshift=-2pt]B) {$y_{21}(x_{21})$};
			\end{scope}
			
			\begin{scope}[shift={(3.6,0)}]
				\coordinate (U) at (0, 1.2);
				\coordinate (T1) at (-0.3, 0.6);
				\coordinate (T2) at (0.3, 0.6);
				\coordinate (B1) at (-0.3, 0);
				\coordinate (B2) at (0.3, 0);
				
				\draw (T1) -- (T2); \draw (B1) -- (B2);
				\draw (U) -- (T1); \draw (U) -- (T2);
				\draw (T1) -- (B1); \draw (T2) -- (B2); \draw (T1) -- (B2);
				
				\node[pt] at (U) {}; \node[pt] at (T1) {}; \node[pt] at (T2) {}; \node[pt] at (B1) {}; \node[pt] at (B2) {};
				\node[anchor=north, inner sep=1pt, font=\tiny] at ([yshift=-2pt]B1) {$y_{21}(x_{21})$};
			\end{scope}
			
			\begin{scope}[shift={(5.4,0)}]
				\coordinate (U) at (0, 1.2);
				\coordinate (T1) at (-0.3, 0.6);
				\coordinate (T2) at (0.3, 0.6);
				\coordinate (B1) at (-0.3, 0);
				\coordinate (B2) at (0.3, 0);
				
				\draw (T1) -- (T2);
				\draw (U) -- (T1); \draw (U) -- (T2);
				\draw (T1) -- (B1); \draw (T2) -- (B2); \draw (T1) -- (B2); \draw (T2) -- (B1);
				
				\node[pt] at (U) {}; \node[pt] at (T1) {}; \node[pt] at (T2) {}; \node[pt] at (B1) {}; \node[pt] at (B2) {};
				\node[anchor=north, inner sep=1pt, font=\tiny] at ([yshift=-2pt]B1) {$y_{21}(x_{21})$};
			\end{scope}
			
			\begin{scope}[shift={(7.2,0)}]
				\coordinate (U1) at (0, 1.2);
				\coordinate (T1) at (0, 0.6);
				\coordinate (B1) at (-0.3, 0);
				\coordinate (B2) at (0, 0);
				\coordinate (B3) at (0.3, 0);
				
				\draw (B1) -- (B2) -- (B3);
				\draw (U1) -- (T1);
				\draw (T1) -- (B1); \draw (T1) -- (B2); \draw (T1) -- (B3);
				
				\node[pt] at (U1) {}; \node[pt] at (T1) {}; \node[pt] at (B1) {}; \node[pt] at (B2) {}; \node[pt] at (B3) {};
				\node[anchor=north, inner sep=1pt, font=\tiny] at ([yshift=-2pt]B1) {$y_{21}(x_{21})$};
			\end{scope}
		\end{tikzpicture}
		\caption{Local two-layer structures of $G$.}
		\label{fig:layer_structures_mini_grid}
	\end{figure}
	
	\begin{figure}[htbp]
		\centering
		\begin{tikzpicture}[line join=round, line cap=round, semithick]
			\begin{scope}[thin, black!80]
				\draw (-2.5, 2.7) -- (8.1, 2.7);
				\draw (-2.5, 2.1) -- (8.1, 2.1);
				\draw (-2.5, 1.5) -- (8.1, 1.5);
				\draw (-2.5, 0.9) -- (8.1, 0.9);
				\draw (-2.5, 0.3) -- (8.1, 0.3);
				\draw (-2.5, -0.6) -- (8.1, -0.6);
				
				\draw (-2.5, 2.7) -- (-2.5, -0.6);
				\draw (-0.9, 2.7) -- (-0.9, -0.6);
				\draw (0.9, 2.7) -- (0.9, -0.6);
				\draw (2.7, 2.7) -- (2.7, -0.6);
				\draw (4.5, 2.7) -- (4.5, -0.6);
				\draw (6.3, 2.7) -- (6.3, -0.6);
				\draw (8.1, 2.7) -- (8.1, -0.6);
			\end{scope}
			
			\node[font=\small] at (-1.7, 1.8) {$v$};
			\node[font=\small] at (-1.7, 1.2) {$B_1$};
			\node[font=\small] at (-1.7, 0.6) {$B_2$};
			\node[font=\small] at (-1.7, 0) {$B_3$};
			
			\node[font=\small] at (0, 2.4) {(1)};
			\node[font=\small] at (1.8, 2.4) {(2)};
			\node[font=\small] at (3.6, 2.4) {(3)};
			\node[font=\small] at (5.4, 2.4) {(4)};
			\node[font=\small] at (7.2, 2.4) {(5)};
			
			\tikzset{pt/.style={circle, fill, inner sep=0pt, minimum size=3pt}}
			
			\begin{scope}[shift={(0,0)}]
				\coordinate (V) at (0, 1.8);
				\coordinate (B1) at (0, 1.2);
				\coordinate (B21) at (-0.3, 0.6);
				\coordinate (B22) at (0.3, 0.6);
				\coordinate (B31) at (0, 0);
				
				\draw (B21) -- (B22);
				\draw (B21) -- (B31); \draw (B22) -- (B31);
				
				\draw (B1) -- (B21); \draw (B1) -- (B22);
				
				\draw (V) -- (B1);
				
				\node[pt] at (V) {};
				\node[pt] at (B1) {};
				\node[pt] at (B21) {}; \node[pt] at (B22) {};
				\node[pt] at (B31) {};
				\node[anchor=north, inner sep=1pt, font=\tiny] at ([yshift=-2pt]B31) {$y_{21}(x_{21})$};
			\end{scope}
			
			\begin{scope}[shift={(1.8,0)}]
				\coordinate (V) at (0, 1.8);
				\coordinate (B11) at (-0.3, 1.2);
				\coordinate (B12) at (0, 1.2);
				\coordinate (B13) at (0.3, 1.2);
				\coordinate (B21) at (-0.3, 0.75);
				\coordinate (B22) at (0, 0.45);
				\coordinate (B23) at (0.3, 0.75);
				\coordinate (B31) at (0, 0);
				
				\draw (B21) -- (B23) -- (B22) -- cycle;
				\draw (B21) -- (B31); \draw (B22) -- (B31);
				
				\draw (B11) -- (B21); \draw (B12) -- (B22); \draw (B13) -- (B23);
				
				\draw (V) -- (B11); \draw (V) -- (B12); \draw (V) -- (B13);
				
				\node[pt] at (V) {};
				\node[pt] at (B11) {}; \node[pt] at (B12) {}; \node[pt] at (B13) {};
				\node[pt] at (B21) {}; \node[pt] at (B22) {}; \node[pt] at (B23) {};
				\node[pt] at (B31) {};
				\node[anchor=north, inner sep=1pt, font=\tiny] at ([yshift=-2pt]B31) {$y_{21}(x_{21})$};
			\end{scope}
			
			\begin{scope}[shift={(3.6,0)}]
				\coordinate (V) at (0, 1.8);
				\coordinate (B11) at (-0.3, 1.2);
				\coordinate (B12) at (0.3, 1.2);
				\coordinate (B21) at (-0.3, 0.6);
				\coordinate (B22) at (0.3, 0.6);
				\coordinate (B31) at (-0.3, 0);
				\coordinate (B32) at (0.3, 0);
				
				\draw (B21) -- (B22); \draw (B31) -- (B32);
				\draw (B21) -- (B31); \draw (B22) -- (B32); \draw (B21) -- (B32);
				
				\draw (B11) -- (B21); \draw (B12) -- (B22);
				
				\draw (V) -- (B11); \draw (V) -- (B12);
				
				\node[pt] at (V) {};
				\node[pt] at (B11) {}; \node[pt] at (B12) {};
				\node[pt] at (B21) {}; \node[pt] at (B22) {};
				\node[pt] at (B31) {}; \node[pt] at (B32) {};
				\node[anchor=north, inner sep=1pt, font=\tiny] at ([yshift=-2pt]B31) {$y_{21}(x_{21})$};
			\end{scope}
			
			\begin{scope}[shift={(5.4,0)}]
				\coordinate (V) at (0, 1.8);
				\coordinate (B11) at (-0.3, 1.2);
				\coordinate (B12) at (0.3, 1.2);
				\coordinate (B21) at (-0.3, 0.6);
				\coordinate (B22) at (0.3, 0.6);
				\coordinate (B31) at (-0.3, 0);
				\coordinate (B32) at (0.3, 0);
				
				\draw (B21) -- (B22);
				\draw (B21) -- (B31); \draw (B22) -- (B32); \draw (B21) -- (B32); \draw (B22) -- (B31);
				
				\draw (B11) -- (B21); \draw (B12) -- (B22);
				
				\draw (V) -- (B11); \draw (V) -- (B12);
				
				\node[pt] at (V) {};
				\node[pt] at (B11) {}; \node[pt] at (B12) {};
				\node[pt] at (B21) {}; \node[pt] at (B22) {};
				\node[pt] at (B31) {}; \node[pt] at (B32) {};
				\node[anchor=north, inner sep=1pt, font=\tiny] at ([yshift=-2pt]B31) {$y_{21}(x_{21})$};
			\end{scope}
			
			\begin{scope}[shift={(7.2,0)}]
				\coordinate (V) at (0, 1.8);
				\coordinate (B11) at (0, 1.2);
				\coordinate (B21) at (0, 0.6);
				\coordinate (B31) at (-0.3, 0);
				\coordinate (B32) at (0, 0);
				\coordinate (B33) at (0.3, 0);
				
				\draw (B31) -- (B32) -- (B33);
				\draw (B21) -- (B31); \draw (B21) -- (B32); \draw (B21) -- (B33);
				
				\draw (B11) -- (B21);
				
				\draw (V) -- (B11);
				
				\node[pt] at (V) {};
				\node[pt] at (B11) {};
				\node[pt] at (B21) {};
				\node[pt] at (B31) {}; \node[pt] at (B32) {}; \node[pt] at (B33) {};
				\node[anchor=north, inner sep=1pt, font=\tiny] at ([yshift=-2pt]B31) {$y_{21}(x_{21})$};
			\end{scope}
		\end{tikzpicture}
		\caption{Local three-layer structures of $G$.}
		\label{fig:extended_layer_grid}
	\end{figure}
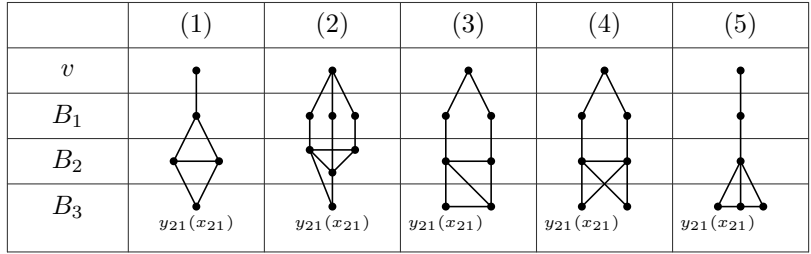

	\begin{lemma}\label{lem:diam34_bound}
		Let $G$ be a kite-saturated graph of order $n\ge 5$. If $3 \le \text{diam}(G) \le 4$, then $e(G) > f_1(n)$.
	\end{lemma}
	
	\begin{proof}
		For convenience, we assume that $G$ is a minimum kite-saturated graph with  $n$ vertices and $3 \le \text{diam}(G) \le 4$. By Lemma \ref{lem:diam34_2connected},  we can assume that $G$ is $2$-connected, which infers $\delta_G\ge 2$. 
		Note that $|G|\ge 5$ by $|K|=5$. Lemma~\ref{lem:n5_exact} and Lemma~\ref{lem:n6_exact} characterizing  all kite-saturated graphs with $e(G)=f_1(n)$  have diameter two, Hence, the statement is true for $n=5,6$. Lemma~\ref{lem:n78_exact} implies that the conclusion is true for $n=7,8$.
		Thus, we now assume $n \ge 9$. Observe that $e(G)\ge f_1(n)$ if $D_2=\emptyset$. Hence, we next consider the case $D_2\ne\emptyset.$

		For short, we  partition $V(G)$ into two disjoint sets $D_2$ and $\overline{D_2}$, where $\overline{D_2}=V(G)\setminus D_2$.   Set $|\overline{D_2}| = n'$. 
		If  $D_2^2\ne\emptyset$, then  Lemma~\ref{lem:D2_structure} implies that $D_2^2$ of $G$ has two different cases.  Thus, $G$ has two possible  structures $(a)$ and $(b)$ as shown in Figure 9. Let $G'$ be the graph obtained from $G$ by removing all vertices in $D^2$ for $(a)$, or from $G$ by removing all vertices in $D_2$ and adding new edge $xy$ for $(b)$. Clearly, $G'\cong G[\overline{D_2}]$ for $(a)$ or $G'\cong G[\overline{D_2}]+xy$ for $(b)$.
		We now show the following claim.
		
		\setcounter{claim}{0} 
		\begin{claim}\label{g'22}
			$G'$ is $2$-connected  with $\text{diam}(G')=2$.
		\end{claim}
		\begin{proof}
			
			We first show that $G'$ is $2$-connected. Assume to the contrary that $G'$ has a cut vertex $v$ such that $G'-v$ contains at least two components. In fact, we claim that $v$ is not a neighbor of some vertex in  $D_2$. Otherwise, we can derive that $v$ is also a cut vertex of $G$, a contradiction.  In order to prove the claim,  the following fact is necessary.
			
				\medskip
			\noindent{\bf Fact 1.} For any  vertex $u\in D_2$,  all vertices of $N(u)\setminus D_2$ belong to one component of $G'-v$.
			
			\proof  We first consider the case $u\in D_2^2$. If  $E(G[D_2^2])\ne\emptyset$, then Fact 1 is trivial since only one neighbor of $u$ belongs to $V(G')$. If $E(G[D_2^2])=\emptyset$, then their common neighbors are adjacent in $G'$ by the definition of $G'$. Hence, conclusion holds.  If $u\in D_2^1$,  then from the definition of $D_2^1$ and Lemma \ref{lem:D2_structure}, its two neighbors are adjacent in $G'$. Recall that $v$ is a neighbor of some vertex of $D_2$. Hence, they are also  adjacent in $G'-v$. Thus, Fact 1 is true. \qed
			
			Let $H_1$ and $H_2$ be two distinct components of $G'-v$. If at most one of $H_1$ and $H_2$ does not contains any neighbor of some vertex in $D_2$, then $v$ is also a cut vertex of $G$. We now assume that $H_1$ and $H_2$ both contain neighbors of some vertex in $D_2$. From Fact 1, we also deduce that $G$ contains a cut vertex $v$, a contradiction. Consequently, $G'$ is $2$-connected, as required.

			We next show $\text{diam}(G')=2$. Let $w_1$ and $w_2$ be two non-adjacent vertices of $G'$. Then the same is true in $G$. Since $G$ is kite-saturated,  by adding $w_1w_2$ to $G$, we get   $K^{w_1w_2}$ as a copy of $K$. Observe that $d_G(w_1),d_G(w_2)\ge 3$, together with Corollary \ref{cor:type_II} which infers that $d_{G}(w_1, w_2) = 2$. Assume thus that $w_3$ is a common neighbor of $w_1$ and $w_2$ in $G$. We next show $w_3\not\in D_2$. In fact, if $w_3\in D_2$, then $w_1,w_2\in N_G(D_2)$, which infers that $w_1w_2\in E(G')$. This leads to a contradiction to the choice of $w_1$ and $w_2$.  We thus conclude that $w_3\in V(G')$, which implies that  $ d_{G'}(w_1, w_2)= d_{G}(w_1, w_2)$.
		\end{proof} Hence, combining Claim~\ref{g'22} and Theorem~\ref{thm:diam2_bound},  we derive that $e(G')\ge 2n' - 5$. If $D^2_2=\emptyset$, then we can obtain that $e(G)=e(G')+2|D^1_2|\ge 2n'-5+2(n-n')=2n-5\ge f_1(n)$ for $n\ge 9$. Hence, we assume that $D^2_2\ne\emptyset$. From Lemma \ref{lem:D2_structure},  the following two cases are now considered.
		
		
			\medskip
		\noindent\textbf{Case 1.} $G$ satisfies the structure (a).
		In this case, it follows from $ G'$ that each vertex in $D^1_2 $ contributes exactly two edges to \( e(G) \), whereas each vertex in \( D_2^2 \) contributes \( \frac{3}{2} \) edges. Consequently, the size of \( G \) is estimated as follows:
		\[
		e(G) = e(G') + 2|D^1_2| + \frac{3}{2}|D_2^2| \ge (2n' - 5) + 2(n - n') - 1 =2n-6.
		\]

			\medskip
		\noindent\textbf{Case 2. }$G$ satisfies the structure (b).
		Observe that,  under $G'$,  each vertex of $D_2 $ contributes exactly two edges to \( e(G) \).  Hence, the size of $G$ satisfies
		\[
		e(G) = e(G') + 2|D_2| -1 \ge (2n' - 5) + 2(n - n')-1=2n-6.
		\]
		Combining above two cases, we conclude that $e(G)\ge 2n-6$. We derive that $e(G)\ge 2n-6>f_1(n)$ if $n\ge 11$.  Bearing in mind the condition $D_2^2\ne\emptyset$ and Lemma~\ref{lem:n9},  we deduce that  $e(G)>f_1(n)$ for $n=9,10$.
		Consequently, we verify our conclusion, as required.
	\end{proof}
	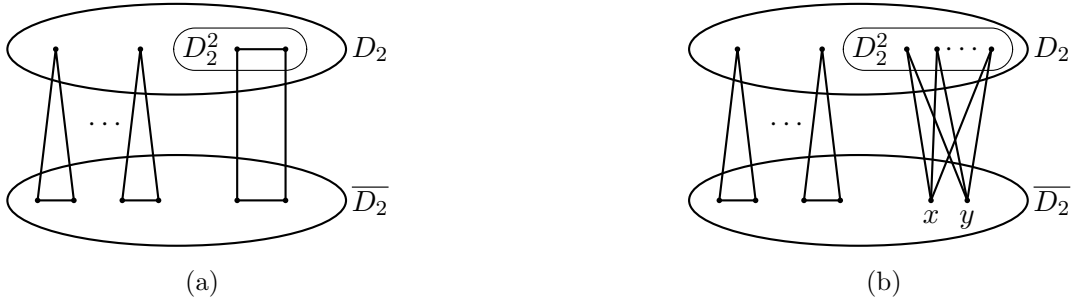
\begin{figure}[H]
		\centering
		\begin{subfigure}[t]{0.48\textwidth}
			\centering
			\begin{tikzpicture}[scale=0.8, line join=round, line cap=round]
				
				\draw[black, thick] (0,2.5) ellipse (2.8 and 0.75);
				\node[black] at (3.2,2.5) {$D_2$};
				
				\draw[black, thick] (0,0) ellipse (2.8 and 0.75);
				\node[black] at (3.2,0) {$\overline{D_2}$};
				
				\coordinate (T1a) at (-2.0,2.5);
				\coordinate (T1b) at (-2.3,0);
				\coordinate (T1c) at (-1.7,0);
				
				\draw[black, thick] (T1a) -- (T1b) -- (T1c) -- cycle;
				
				\node[black] at (-1.15,1.25) {$\cdots$};
				
				\coordinate (T2a) at (-0.6,2.5);
				\coordinate (T2b) at (-0.9,0);
				\coordinate (T2c) at (-0.3,0);
				
				\draw[black, thick] (T2a) -- (T2b) -- (T2c) -- cycle;
				
				\coordinate (C1) at (1.0,2.5);
				\coordinate (C2) at (1.8,2.5);
				\coordinate (C3) at (1.0,0);
				\coordinate (C4) at (1.8,0);
				
				\node[black] at (0.4, 2.5) {$D_2^2$};
				\draw[black, thin, rounded corners=8pt] (-0.05, 2.15) rectangle (2.15, 2.85);
				
				\draw[black, thick] (C1) -- (C2);  
				\draw[black, thick] (C1) -- (C3);  
				\draw[black, thick] (C2) -- (C4);  
				\draw[black, thick] (C3) -- (C4);  
				
				\foreach \p in {T1a,T2a,C1,C2}
				\fill[black] (\p) circle (1.3pt);
				
				\foreach \p in {T1b,T1c,T2b,T2c,C3,C4}
				\fill[black] (\p) circle (1.3pt);
				
			\end{tikzpicture}
			\caption{}
		\end{subfigure}
		\hfill
		\begin{subfigure}[t]{0.48\textwidth}
			\centering
			\begin{tikzpicture}[scale=0.8, line join=round, line cap=round]
				
				\draw[black, thick] (0,2.5) ellipse (2.8 and 0.75);
				\node[black] at (3.2,2.5) {$D_2$};
				
				\draw[black, thick] (0,0) ellipse (2.8 and 0.75);
				\node[black] at (3.2,0) {$\overline{D_2}$};
				
				\coordinate (T1a) at (-2.0,2.5);
				\coordinate (T1b) at (-2.3,0);
				\coordinate (T1c) at (-1.7,0);
				
				\draw[black, thick] (T1a) -- (T1b) -- (T1c) -- cycle;
				
				\node[black] at (-1.15,1.25) {$\cdots$};
				
				\coordinate (T2a) at (-0.6,2.5);
				\coordinate (T2b) at (-0.9,0);
				\coordinate (T2c) at (-0.3,0);
				
				\draw[black, thick] (T2a) -- (T2b) -- (T2c) -- cycle;
				
				\coordinate (U1) at (0.8,2.5);
				\coordinate (U2) at (1.3,2.5);
				\coordinate (U3) at (2.2,2.5);
				
				\node[black] at (0.2, 2.5) {$D_2^2$};
				\draw[black, thin, rounded corners=8pt] (-0.25, 2.15) rectangle (2.55, 2.85);
				
				\node[black] at (1.75,2.5) {$\cdots$};
				
				\coordinate (L1) at (1.2,0);
				\coordinate (L2) at (1.8,0);
				
				\draw[black, thick] (U1) -- (L1);
				\draw[black, thick] (U1) -- (L2);
				\draw[black, thick] (U2) -- (L1);
				\draw[black, thick] (U2) -- (L2);
				\draw[black, thick] (U3) -- (L1);
				\draw[black, thick] (U3) -- (L2);
				
				\foreach \p in {T1a,T2a,U1,U2,U3}
				\fill[black] (\p) circle (1.3pt);
				
				\foreach \p in {T1b,T1c,T2b,T2c,L1,L2}
				\fill[black] (\p) circle (1.3pt);
				
				\node[black, below] at (L1) {$x$};
				\node[black, below] at (L2) {$y$};
				
			\end{tikzpicture}
			\caption{}
		\end{subfigure}
		\vspace{-2mm}
		\caption{$V(G)$ with parts $D_2$ and $\overline{D_2}$.}
	\end{figure}
	
	\begin{proof}[\bf Proof of Theorem 3.1]
		Let $G$ be a minimum connected kite-saturated graph of order $n \ge 5$. By Lemma~\ref{lem:diam_4}, we get that the diameter of $G$ is restricted to $2 \le \text{diam}(G) \le 4$ (as $\text{diam}(G)=1$ implies $G \cong K_n$, which contains $K$ for $n \ge 5$). We proceed by analyzing the diameter and connectivity of $G$:
		
			\medskip
		\noindent\textbf{Case 1. }$3 \le \text{diam}(G) \le 4$.
		Lemma~\ref{lem:diam34_2connected} show that any kite-saturated graph is $2$-connected  if its diameter equals to three  or four. Furthermore,  Lemma~\ref{lem:diam34_bound} implies that $e(G) > f_1(n)$ for all $n \ge 5$.
		
			\medskip
		\noindent\textbf{Case 2.}  $\text{diam}(G) = 2$. We now assume that $G$ is connected and contains a cut vertex. By Lemma~\ref{lem:1conn_exact}, we have $e(G) \ge f_1(n)$, and equality holds only if $G \cong F_n^{\mathrm{odd}}$ (for odd $n$) or $G \cong F_n^{\mathrm{even}}$ (for even $n$).  If  $G$ is $2$-connected, then, by Lemmas  \ref{lem:n5_exact}, \ref{lem:n6_exact}, \ref{lem:n78_exact} and \ref{lem:diam2_n9}, we prove that $e(G) \ge f_1(n)$, moreover, equality holds only if $G \cong H_5$ for $n = 5$, or $G \in \{F_6^1, F_6^2, F_6^3\}$ for $n=6$.
		
		Therefore,  we are done.
	\end{proof}

	\section{Disconnected kite-saturated graphs}
	
	Let $G$ be a minimum kite-saturated graph of order $n$. If $G$ is disconnected, then let $G \cong H_1 \cup H_2 \cup \cdots \cup H_m$ with $m \ge 2$, where each $H_i$ is a component of $G$.
	
	\begin{lemma}\label{lem:comp_order}
		$H_i \cong K_{|H_i|}$ if $|H_i| \le 4$, $H_i$ is kite-saturated otherwise. Furthermore, every component of $G$ has order $4$, with at most one exception.
	\end{lemma}
	
	\begin{proof}
		We first assume  $|H_i| \le 4$. If $H_i \not\cong K_{|H_i|}$, then there exists non-adjacent vertices $u, v \in V(H_i)$.  But $G+ uv$ cannot create a copy of $K$, we get a contradiction to  the assumption that $G$ is kite-saturated. If $|H_i| \ge 5$, then $H_i\not\cong K_{|H_i|}$, otherwise, $H_i$ contains a copy of $K$, which contradicts that $G$ is kite-saturated. Let $x$ and $y$ be two vertices of $H_i$ with $xy\not\in E(G)$.  So $G+xy$ contains a copy of $K$ having the added edge $xy$, the same holds for $H_i+xy$. Hence, $H_i$ is kite-saturated.
		
		We next show that at most one component of $G$ has order other than $4$. Suppose, by contradiction, that there exist at least two components $H_i$ and $H_j$ such that $|H_i| \neq 4$ and $|H_j| \neq 4$. We consider the following cases
		
			\medskip
		\noindent\textbf{Case 1. }$|H_i| < 4$ and $|H_j| < 4$. By the argument above, we get  that $H_i$ and $H_j$ are complete graphs and do not contain $4$-cycles.  Observe that, by adding any edge between $H_i$ and $H_j$, this can not create a $4$-cycle. it follows that $G$ is not kite-saturated, a contradiction.
		
		\medskip
		\noindent\textbf{Case 2.} $|H_i| > 4$ and $|H_j| > 4$.
		By the saturation of these two components, they contain vertices with degree at least three. Let $x_i$ and $x_j$ be the vertices of $H_i$ and $H_j$ with degree at least three, respectively.
		Consider $G + x_ix_j$. Since $d_G(x_i,x_j) = \infty \ge 3$, by Observation~\ref{obs:dist_type}, the added edge $x_ix_j$ is the type I edge of $K^{x_ix_j}$. Conversely,  since $d_G(x_i) \ge 3$ and $d_G(x_j) \ge 3$, neither $x_i$ nor $x_j$ can serve as the type $\alpha_1$ vertex of $K^{x_ix_j}$, otherwise, we will get a copy of $K$, a contradiction.  This contradicts our assumption.
		
		\medskip
		\noindent\textbf{Case 3. }$|H_i| > 4$ and $|H_j| < 4$.
		By the same argument in Case 2, $H_i$ contains a vertex $x_i$ with $d_G(x_i) \ge 3$. For any $x_j \in V(H_j)$, $G + x_ix_j$ will create $K^{x_ix_j}$. Observation~\ref{obs:dist_type} together with $d_G(x_i,x_j) \ge 3$, implies that $x_ix_j$ is the type I edge. Clearly,  $x_i$ is not the type $\alpha_1$ vertex. Thus, $x_j$ is the type $\alpha_1$ vertex, it follows that $x_j$ is contained in a $4$-cycle of $H_j$. But it is impossible.
		
		Therefore, we are done.
	\end{proof}
	
	\section{The proof of Theorem \ref{thm:main} }

	\begin{proof}[\bf Proof of Theorem \ref{thm:main}]
		Let $G$ be a minimum kite-saturated graph of order $n \ge 5$. Observation \ref{obs:upperbound} implies that $e(G)\le f(n)$. To establish our conclusion, in the remaining part of the proof, we will show the  lower bound  $e(G)\ge f(n)$.
		
		We first assume that $G$ is connected. Evidently, $f_1(n)\ge f(n)$. Theorem~\ref{thm:connected_bound} brings that $e(G)\ge f_1(n)\ge f(n)$, with equality only if $n\neq 4k+2$ with $k\ge 1$ and $G\in\{H_5, F_{n}^{\mathrm{odd}},F_{n}^{\mathrm{even}} \}$.
		We now assume that $G$ is disconnected. By Lemma~\ref{lem:comp_order}, $G$ has $m$ components $H_1, H_2, \dots, H_m$, where at most one component has order other than $4$, and every other component is isomorphic to $K_4$. Without loss of generality, assume that $H_i\cong K_4$ for $i\le m-1$. Let $p = m-1$. then $n = |H_m| + 4 p$. We now estimate $e(G)$ depending on $n$ modulo $4$.
		
		\medskip
		\noindent\textbf{Case 1. }$n \equiv 0 \pmod 4$.
		If all components of $G$ are $K_4$ components, then  $e(G) = \frac{6n}{4} = \frac{3n}{2} > f(n)$. We thus assume that $G$ contains one component of order different from $4$  Then  $|H_m|\equiv 0 \pmod 4$ and $|H_m| \ge 8$.  By Theorem 3.1, $e(H_m) \ge f_1(|H_m|) = \frac{3|H_m|}{2} - 1$, and then equality holds only if $G\cong F_{n-4p}^{\mathrm{even}}$. Consequently,
		\[ e(G) = e(H_m) + 6p \ge \left( \frac{3(n-4p)}{2} - 1 \right) + 6p = \frac{3n}{2} - 1 = f(n). \]
		Moreover, equality holds only if $H_i \cong F_{n-4p}^{\mathrm{even}}$, so $G \in \{F_{n-4p}^{\mathrm{even}} \cup pK_4 \mid 1 \le p \le \frac{n-8}{4}\}$.
		
		\medskip
		\noindent\textbf{Case 2.} $n \equiv 1 \pmod 4$. In this case, $G$ contains one component of order different from $4$. Then $|H_m| \equiv 1 \pmod 4$.
		If $|H_m| = 1$, then $e(G) = 0 + 6(\frac{n-1}{4}) = \frac{3n-3}{2} = f(n)$, which infers that $G \cong K_1 \cup \frac{n-1}{4}K_4$.
		If $|H_m| \ge 5$, then, by Theorem 3.1, $e(H_m) \ge f_1(|H_m|) = \frac{3|H_m|-3}{2}$ with equality only if $H_m \in \{H_5, F_{n-4p}^{\mathrm{odd}}\}$. We thus get
		\[ e(G) \ge \frac{3(n-4p)-3}{2} + 6p = \frac{3n-3}{2} = f(n). \]
		Furthermore, equality holds only if $G \in \{F_{n-4p}^{\mathrm{odd}} \cup pK_4, \mid 1 \le p \le \frac{n-5}{4}\}\cup \{H_5\cup \frac{n-5}{4}K_4\}$.
		
		\medskip
		\noindent\textbf{Case 3.} $n \equiv 2 \pmod 4$. Clearly,  $H_m$ has $|H_m| \equiv 2\pmod 4$.
		If $|H_m| = 2$, then $e(G) = 1 + 6(\frac{n-2}{4}) = 1 + \frac{3n-6}{2} = \frac{3n-4}{2} = \frac{3n}{2} - 2 = f(n)$, so $G \cong K_2 \cup \frac{n-2}{4}K_4$.
		If $|H_m| \ge 6$, then by Theorem 3.1, $e(H_m) \ge f_1(|H_m|) = \frac{3|H_m|}{2} - 1$.  Based on them, we have that
		\[ e(G) \ge \frac{3(n-4p)}{2} - 1 + 6p = \frac{3n}{2} - 1 > f(n). \]
		
		\medskip
		\noindent\textbf{Case 4. }$n \equiv 3 \pmod 4$. We observe that $H_m$ has $|H_m| \equiv 3\pmod 4$.
		If $|H_m| = 3$, then $e(G) = 3 + 6(\frac{n-3}{4}) = 3 + \frac{3n-9}{2} = \frac{3n-3}{2} = f(n)$, so $G \cong K_3 \cup \frac{n-3}{4}K_4$.
		If $|H_m| \ge 7$, then, by Theorem 3.1, $e(H_m) \ge f_1(|H_m|) = \frac{3|H_m|-3}{2}$ with equality only if $H_m \cong F_{n-4p}^{\mathrm{odd}}$. We thus have
		\[ e(G) \ge \frac{3(n-4p)-3}{2} + 6p = \frac{3n-3}{2} = f(n), \]
		and equality holds if and only if $G \in \{F_{n-4p}^{\mathrm{odd}} \cup pK_4, \mid 1 \le p \le \frac{n-7}{4}\}\cup \{K_3\cup \frac{n-3}{4}K_4\}$.
		
		Combining the above argument, we conclude that $\sat(n,K)=f(n)$ and $\Sat(n,K)= \mathcal{K}_n$.
	\end{proof}

	\noindent{\bf Acknowledgments}
	
	Qing Cui is partially supported by the National Natural Science Foundation of China (No. 12171239). Shengjin Ji is partially by the National Natural Science Foundation of China (No. 12561061).
	Fuhong Ma is partially supported by  the National Natural Science Foundation of China (No.12401458) and the Natural Science Foundation of Shandong Province, China  (No.ZR2024QA239).

\end{document}